\documentclass[11pt,a4paper]{article}

\usepackage[utf8]{inputenc}
\usepackage[T1]{fontenc}
\usepackage[english]{babel}
\usepackage{lmodern}
\usepackage{float}
\usepackage{amsmath,amssymb,amsthm,amsfonts}
\numberwithin{equation}{section}
\usepackage{mathtools}
\usepackage{bm}

\usepackage[margin=2.5cm]{geometry}
\usepackage{enumitem}
\usepackage{microtype}

\usepackage[round,authoryear]{natbib}
\usepackage[colorlinks=true,
            citecolor=blue,
            linkcolor=blue,
            urlcolor=blue]{hyperref}

\theoremstyle{plain}
\newtheorem{theorem}{Theorem}[section]

\newtheorem{proposition}[theorem]{Proposition}

\theoremstyle{definition}
\newtheorem{definition}[theorem]{Definition}
\newtheorem{example}[theorem]{Example}

\theoremstyle{remark}
\newtheorem{remark}[theorem]{Remark}

\newcommand{\R}{\mathbb{R}}
\newcommand{\N}{\mathbb{N}}

\newcommand{\Rp}{\R_{+}}

\newcommand{\E}{\mathbb{E}}

\newcommand{\im}{\mathrm{i}}                 % imaginary unit
\newcommand{\ip}[2]{\langle #1,\, #2 \rangle} % scalar product
\newcommand{\charf}{\varphi}                  % characteristic function

\newcommand{\Mplus}{\mathcal{M}_{+}}
\newcommand{\Sclass}{\mathcal{S}}             % self-decomposable class / Schwartz
\DeclareMathOperator{\Hess}{Hess}
\begin{document}
% =====================================================================

\title{A Kernel Approach to Multivariate Generalized Gammas Convolutions}

\author{%
  L\'eo Gonin\thanks{Corresponding author: \texttt{leo.gonin@univ-lyon1.fr}.} \\
  \and
  Esterina Masiello\thanks{Corresponding author: \texttt{esterina.masiello@univ-lyon1.fr}.} \\
  \and
  V\'eronique Maume-Deschamps\thanks{Corresponding author: \texttt{veronique.maume@univ-lyon1.fr}.} \\
  \\
}
\date{\today}

\maketitle
  \begin{center}
  \normalsize Université Lyon 1, Centrale Lyon, INSA Lyon, Université Jean Monnet, CNRS, ICJ UMR5208, 69622 Villeurbanne, France. \\
  \end{center}

\begin{abstract}
We propose a kernel-based estimation framework for Multivariate
Generalized Gamma convolutions (MGGC), a class of probability
distributions on $\R_{+}^{d}$. MGGC are useful for risk modeling e.g. Their densities are in general not explicit. The construction relies on a new Stein
operator for the MGGC class, derived from a partial
differential equation satisfied by the characteristic function and
expressed as an integral against the underlying Thorin measure. This leads to construct a new kernel and an estimation procedure in the derived reproducing kernel Hilbert space (RKHS); with consistency results. Numerical
experiments illustrate the practical behaviour of the proposed
estimator.
\end{abstract}
\bigskip


% =====================================================================
\section*{Introduction}
\addcontentsline{toc}{section}{Introduction}
% =====================================================================

Multivariate Generalized Gamma Convolutions (MGGC) form a rich class of infinitely divisible
distributions introduced by Thorin and later
developed systematically by Bondesson
\citep{Thorin1977Pareto,Thorin1977Lognormal,Bondesson1992}. Their
multivariate extensions provide flexible models on $\R_{+}^{d}$ and are
particularly suited to the modeling of positive dependent quantities, useful e.g for risk modeling \cite{LavernyFerrieroNisipasu2022}.
However, statistical inference within this class remains difficult. Except
in specific cases, the density of a MGGC is not explicit,
and existing estimation procedures often rely either on projections from a
known density or on series expansions whose numerical implementation becomes
delicate in high dimension
\citep{MilesFurmanKuznetsov2021,LavernyEtAl2021,Laverny2022Projections}. The aim of this paper is to propose an alternative estimation framework for MGGC distributions which avoids the explicit evaluation of the density. The
main idea is to derive a Stein operator directly from the equation satisfied
by the characteristic function of a MGGC. This operator is expressed in
terms of the Thorin measure and therefore remains naturally adapted to both
discrete and continuous Gamma convolutions. We combine this operator
with RKHS techniques in order to construct a
Kernel Stein Discrepancy tailored to the MGGC class. The paper is organised as follows. Section~\ref{sec:background-ggc} recalls
the definition and main properties of generalized gamma convolutions. Sections~\ref{sec:background-rkhs-kme} and~\ref{sec:background-stein}
review the RKHS, kernel mean embedding, Stein operator and Kernel Stein Discrepancy tools used throughout the article. In Section~\ref{sec:stein-ggc}, we derive the Stein operators associated to MGGC distributions and relate them to the underlying L\'evy and
gamma subordinator representations. Section~\ref{sec:ksd-ggc} introduces
the corresponding Kernel Stein Discrepancy, establishes its main theoretical
properties and studies the associated minimum-discrepancy estimator. Finally,
Section~\ref{sec:numerical-applications} presents numerical experiments
illustrating the behaviour of the proposed method in univariate,
multivariate and higher-dimensional settings.
% =====================================================================
\section{Background on MGGC}
\label{sec:background-ggc}
% =====================================================================
The Generalized Gamma Convolutions were introduced by \citet{Thorin1977Lognormal} to establish the infinite divisibility of the lognormal and Pareto distributions, and were later developed in depth by \citet{Bondesson1992}; their multivariate counterpart (MGGC) extends this construction to $\mathbb{R}_+^d$, preserving infinite divisibility and an interpretable parametrization through Thorin measures. Despite the interest of the MGGC class, the statistical literature on the estimation of the MGGC class is sparse.
\citet{FurmanHackmannKuznetsov2017}, \citet{Miles2023Thesis} and \citet{MilesFurmanKuznetsov2021}
proposed projection procedures from a known density onto the class of finite
gamma convolutions. \citet{LavernyEtAl2021} bridged the
multivariate gap by exploiting tensorised Laguerre expansions and provided
the first consistent estimation procedure in dimension $d \geq 1$;
\citet{Laverny2022Projections} subsequently scaled the methodology up to
high-dimensional regimes through random projections of the dataset. These two papers give consistent approximations of MGGC densities but require a very careful implementation with arbitrary precision. More
recent contributions concerning the analytic behaviour of the class
include the work of \citet{Sayit2024}, who proves that, within the MGGC class,
weak convergence already implies convergence of the mean -- a robustness
property of independent interest in financial applications for example.

We propose a novel approach to  estimate Mulivariate Generalized Gammas Convolutions using RKHS and Stein operator. It leads to the estimation of the parameters and a better understanding of the MGGC structure.

% ---------------------------------------------------------------------
\subsection{Multivariate generalized Gamma convolutions}
\label{sec:ggc-definition}
% ---------------------------------------------------------------------

Throughout the paper, $d \in \N^{*}$ is fixed and
$X = (X_{1},\dots,X_{d})$ denotes a random vector (r.v) with values in
$\Rp^{d}$. As is customary in the multivariate setting, we abbreviate
$X = X_{1}$ when $d = 1$. The characteristic function of a random variable $Z$ is denoted $\varphi_Z(t), \ t\in \mathbb{R}^d$. Recall that
$Z \sim \Gamma(\alpha,s)$ with shape parameter $\alpha > 0$ and scale
parameter $s > 0$, admits the characteristic function
\begin{equation*}
  \charf_{X}(t) \;=\; \bigl(1 - \im\, s\, t\bigr)^{-\alpha},
  \qquad t \in \R.
\end{equation*}
Generalized gamma convolutions are obtained as weak limits of finite
convolutions of independent Gamma variables, possibly with different shapes
and scales \citep{Thorin1977Pareto,Thorin1977Lognormal,Bondesson1992}. The
multivariate counterpart, originally proposed by~\citet{Bondesson2009} and
formalised in~\citep{LavernyEtAl2021,Laverny2022Projections}, is built in the
same way using \emph{vector-valued} scales.

The framework relies on a positive measure $\nu$ defined on the positive
orthant, called the \emph{Thorin measure}.

\begin{definition}[Thorin measure]
\label{def:thorin-measure}
A \emph{Thorin measure} on $\Rp^{d}$ is a positive Radon measure
$\nu \in \Mplus(\Rp^{d})$ satisfying the integrability conditions
\begin{equation*}
  \int_{\{\|s\| \leq 1\}} \bigl|\log\|s\|\bigr|\,\nu(\mathrm{d}s) < \infty
  \qquad\text{and}\qquad
  \int_{\{\|s\| > 1\}} \frac{1}{\|s\|}\,\nu(\mathrm{d}s) < \infty.
\end{equation*}
\end{definition}

\begin{definition}[Multivariate GGC: classes $\mathcal{G}_{d,n}$ and $\mathcal{G}_{d}$]
\label{def:Gd-classes}
Let $X$ be a $d$-variate random vector with values in $\Rp^{d}$.

Let $\nu$ be a Thorin measure on $\Rp^{d}$ in the sense of
Definition~\ref{def:thorin-measure}. We say that $X$ is a
\emph{multivariate generalized gamma convolution} with Thorin
measure $\nu$, written $X \sim \mathcal{G}_{d}(\nu)$, if its
log-characteristic function writes
\begin{equation}
\log \charf_{X}(t) \;=\; -\!\int_{\Rp^{d}}\! \log\bigl(1 - \im\,\ip{s}{t}\bigr)\,\nu(\mathrm{d}s),
\qquad t \in \R^{d}.
\label{eq:cf-Gd}
\end{equation}
We denote by $\mathcal{G}_{d}$ the class of such distributions, also
called the \emph{multivariate Thorin class}.
\end{definition}
\begin{example}[Multivariate gamma convolutions and the $X = SZ$ representation]
\label{ex:gamma}
Let $X$ be a multivariate GGC whose Thorin measure is the atomic measure
\[
\nu = \sum_{i=1}^{n} \alpha_{i},\delta_{S_{\cdot,i}},
\]
where $\alpha = (\alpha_{1},\dots,\alpha_{n}) \in \mathbb{R}_{+}^{n}$ and
$S  \in \mathbb{R}_{+}^{d\times n}$. Then
\(
X \sim \mathcal{G}_{d,n}
\)
admits the representation
\(X = SZ\)
where $Z=(Z_{1},\dots,Z_{n})$ is a vector of mutually independent Gamma random variables satisfying
\(
Z_i \sim \Gamma(\alpha_i,1)
.
\)
\end{example}

\begin{example}[Pareto and lognormal distributions]
\label{ex:pareto-lognormal}
Pareto and lognormal distributions both belong to $\mathcal{G}_{1}$
\citep{Thorin1977Pareto,Thorin1977Lognormal}; they correspond to  non-trivial Thorin
measures absolutely continuous wrt Lebesgue measure.
\end{example}
% ---------------------------------------------------------------------
\subsection{Identifiability}
\label{sec:ggc-identifiability}
% ---------------------------------------------------------------------

The parametric models introduced in
Definition \ref{def:Gd-classes} serve as
targets for the Stein operator and the Kernel Stein Discrepancy
constructed below. The asymptotic guarantees of any
minimum-discrepancy estimator built on those models, however, only
translate into consistency of the parameter when the parametrisation
itself is injective. We therefore close the present background section
by recording the identifiability properties of the classes
$\mathcal{G}_{d}$ and $ \mathcal{G}_{d,n}$
and by exhibiting an explicit constrained parameter space on which
$\mathcal{G}_{d,n}$ becomes identifiable.

\begin{proposition}[Identifiability of $\mathcal{G}_{d}$]
\label{prop:gd-identifiable}
The map $\nu \mapsto \mathbb{P}_{\nu}$ is injective on the set of Thorin
measures of Definition~\ref{def:thorin-measure}: for any two Thorin
measures $\nu_{1}, \nu_{2}$ on $\Rp^{d}$,
\[
  \mathbb{P}_{\nu_{1}} \;=\; \mathbb{P}_{\nu_{2}}
  \;\Longrightarrow\;
  \nu_{1} \;=\; \nu_{2}.
\]
\end{proposition}

\begin{proof}
The characteristic function $\charf_{X}$ uniquely determines the law of
$X$ on $\Rp^{d}$, so it suffices to show that
$\charf_{\nu_{1}} = \charf_{\nu_{2}}$ implies $\nu_{1} = \nu_{2}$.
The right-hand side of~\eqref{eq:cf-Gd} is a L\'evy--Khintchine exponent
whose L\'evy measure is in one-to-one correspondence with $\nu$ through
the change of variable $u = \ip{s}{t}$;
see~\citet[Theorem~3.1.1]{Bondesson1992} for the univariate case
and~\citet[Section~2]{LavernyEtAl2021} for the multivariate extension.
Hence $\charf_{\nu_{1}} = \charf_{\nu_{2}}$ implies $\nu_{1} = \nu_{2}$.
\end{proof}

For the discrete sub-classes $\mathcal{G}_{d,n}$, the Thorin measure is purely atomic with $n$ atoms,
\begin{equation*}
  \nu_{(\alpha,s)} \;=\; \sum_{j=1}^{n} \alpha_{j}\, \delta_{s_{j}},
  \qquad
  (\alpha, s) \in \R^{n}_{*} \times (\Rp^{d})^{n}
\end{equation*} 
The characterization of the Thorin measure in this case becomes the pair of parameters
$(\alpha, s)$. The mapping from $(\alpha, s)$ to $\nu_{(\alpha,s)}$, and
hence to $\mathbb{P}_{(\alpha,s)}$, is therefore subject to two
structural ambiguities:
\begin{enumerate}[label=(\roman*)]
  \item \emph{Permutation invariance.} For any permutation $\sigma$ of
        $\{1,\dots,n\}$,
        \[
          \sum_{j=1}^{n} \alpha_{j}\, \delta_{s_{j}}
          \;=\;
          \sum_{j=1}^{n} \alpha_{\sigma(j)}\, \delta_{s_{\sigma(j)}},
        \]
        so that $(\alpha, s)$ and $(\alpha_{\sigma}, s_{\sigma})$ yield
        the same distribution.
  \item \emph{Atom fusion.} If $s_{j} = s_{\ell}$ for some
        $j \neq \ell$, then
        $\alpha_{j}\,\delta_{s_{j}} + \alpha_{\ell}\,\delta_{s_{\ell}}
         = (\alpha_{j} + \alpha_{\ell})\,\delta_{s_{j}}$,
        so that $\nu_{(\alpha,s)}$ admits an equivalent representation
        with $n-1$ atoms only. The parametrisation by $n$ atoms is then
        redundant.
\end{enumerate}

In particular, $(\alpha, s) \mapsto \mathbb{P}_{(\alpha,s)}$ is not
injective on $\R^{n}_{*} \times (\Rp^{d})^{n}$, and parameters identifiability
of $\mathcal{G}_{d,n}$ requires further restrictions on the parameter
space. We now exhibit a natural and compact restriction.

\begin{definition}[Constrained parameter space]
\label{def:theta-constrained}
Fix constants
$0 < \underline{\alpha} \leq \overline{\alpha} < \infty$,
$0 < \underline{s} \leq \overline{s} < \infty$, and $\delta > 0$. Consider the compact set:
\[
  \Theta_c \;\coloneqq\; \Bigl\{\,
     (\alpha, s) \in
     [\underline{\alpha},\overline{\alpha}]^{n}
     \times
     [\underline{s},\overline{s}]^{n \times d}
     \;:\;
     \|s_{j} - s_{\ell}\| \geq \delta \;\,\forall\, j \neq \ell, \;\;
     s_{j+1,1} - s_{j,1} \geq \delta \;\,\forall\, j = 1,\dots,n-1
  \,\Bigr\}.
\]
\end{definition}

The constraint $\|s_{j} - s_{\ell}\| \geq \delta$ rules out atom fusion
by enforcing that $\nu_{(\alpha,s)}$ has exactly $n$ pairwise distinct
atoms, while the constraint $s_{j+1,1} - s_{j,1} \geq \delta$ fixes a
canonical ordering of the atoms by their first coordinate and removes
the permutation invariance. Note that the components $\alpha_{j}$ are
not required to be distinct: once the positions of the atoms are
uniquely labelled, the masses are automatically identified. We can see that this space is compact.

\begin{proposition}[Parameters identifiability of $\mathcal{G}_{d,n}$ on $\Theta_c$]
\label{prop:gdn-identifiable}
The map $(\alpha, s) \mapsto \mathbb{P}_{(\alpha,s)}$ is injective on the
set $\Theta$ of Definition~\ref{def:theta-constrained}: for any
$\theta_{1} = (\alpha^{(1)}, s^{(1)})$ and
$\theta_{2} = (\alpha^{(2)}, s^{(2)})$ in $\Theta_c$,
\[
  \mathbb{P}_{\theta_{1}} \;=\; \mathbb{P}_{\theta_{2}}
  \;\Longrightarrow\;
  \theta_{1} \;=\; \theta_{2}.
\]
\end{proposition}

\begin{remark}[Choice of the constants]
\label{rem:identif-constants}
The constants
$\underline{\alpha},\overline{\alpha},\underline{s},\overline{s},\delta$
in Definition~\ref{def:theta-constrained} are merely required to exist;
their numerical values play no role in identifiability and only matter
when compactness of $\Theta_c$ is invoked, e.g.\ for the strong consistency
of the Minimum Stein Discrepancy Estimator~\eqref{eq:msde}. In practice,
$\underline{\alpha}, \underline{s}, \delta$ are taken arbitrarily small
and $\overline{\alpha}, \overline{s}$ arbitrarily large, so that
$\Theta_c$ covers any prescribed compact region of the parameter space
without affecting the conclusion of
Proposition~\ref{prop:gdn-identifiable}.
\end{remark}

% =====================================================================
\section{Background on RKHS and KME}
\label{sec:background-rkhs-kme}
% =====================================================================

\label{sec:rkhs-litterature}

This section presents a short overview on Reproducing Kernel Hilbert Spaces (RKHS) and their usefulness for statistical estimations. The theory of RKHS goes back to the
seminal paper of \citet{Aronszajn1950}, in which the foundational
correspondence between positive-definite kernels and Hilbert spaces of
functions was established.
The use of RKHS to represent probability measures via the so-called
\emph{Kernel Mean Embedding} (KME) was introduced by \citet{SmolaEtAl2007}. They showed that this embedding is injective. This injectivity
property leads to \emph{Maximum Mean Discrepancy} (MMD)
proposed by \citet{GrettonEtAl2012}. 
We shall use the RKHS framework to formulate a
discrepancy between probability measures that does not require the
existence of an explicit density. 
% ---------------------------------------------------------------------
\subsection{Reproducing kernel Hilbert spaces}
\label{sec:rkhs-def}
% ---------------------------------------------------------------------

Throughout this section, $\mathcal{X}$ denotes a non-empty set, typically a subset
of $\R^{d}$. We start by recalling the basic notions; the presentation
follows~\citet{Aronszajn1950,ScholkopfSmola2002,SteinwartChristmann2008}.

\begin{definition}[Reproducing kernel Hilbert space]
\label{def:rkhs}
A function $k \colon \mathcal{X} \times \mathcal{X} \to \R$ is called a \emph{kernel} on
$\mathcal{X}$ if there exist a Hilbert space $(\mathcal{H}, \langle \cdot, \cdot \rangle_{\mathcal{H}})$
and a feature map $\phi \colon \mathcal{X} \to \mathcal{H}$ such that
\begin{equation}
  k(x,y) \;=\; \langle \phi(x), \phi(y) \rangle_{\mathcal{H}},
  \qquad \forall x,y \in \mathcal{X}.
  \label{eq:kernel-feature}
\end{equation}
A Hilbert space $(\mathcal{H}, \langle \cdot, \cdot \rangle_{\mathcal{H}})$ of real-valued
functions on $\mathcal{X}$ is called a \emph{reproducing kernel Hilbert space} (RKHS)
with kernel $k$ if $k(\cdot, x) \in \mathcal{H}$ for every $x \in \mathcal{X}$ and the
\emph{reproducing property} holds:
\begin{equation*}
  f(x) \;=\; \langle f,\, k(\cdot, x) \rangle_{\mathcal{H}},
  \qquad \forall f \in \mathcal{H},\ \forall x \in \mathcal{X}.
\end{equation*}
We denote such a space by $\mathcal{H}_k$ when the dependence on $k$ has to be made
explicit.
\end{definition}

A direct consequence of~\eqref{eq:kernel-feature} is that any kernel $k$
is \emph{symmetric} ($k(x,y) = k(y,x)$) and \emph{positive semi-definite}.
The classical theorem of \citet{Aronszajn1950} states the converse: every
symmetric positive semi-definite function $k$ on $\mathcal{X}\times \mathcal{X}$ is the
reproducing kernel of a unique RKHS, with canonical feature map
$\phi(x) = k(\cdot, x)$. By the reproducing property,
\begin{equation*}
  k(x,y) \;=\; \bigl\langle k(\cdot, x),\, k(\cdot, y) \bigr\rangle_{\mathcal{H}},
\end{equation*}
which is the standard ``kernel trick'': one can compute inner products in
$\mathcal{H}$ without ever evaluating the feature map $\phi$ explicitly.

\begin{example}[Common scalar kernels]
\label{ex:kernels}
For $\mathcal{X} \subseteq \R^{d}$, classical examples include the Gaussian (RBF)
kernel
$k_{\mathrm{RBF}}(x,y) = \exp\!\bigl(-\|x-y\|^{2}/(2\sigma^{2})\bigr)$ with
bandwidth $\sigma > 0$, the Laplace kernel
$k_{\mathrm{Lap}}(x,y) = \exp(-\|x-y\|/\sigma)$, the polynomial kernel
$k_{\mathrm{poly}}(x,y) = (\langle x,y\rangle + c)^{d}$ with $c\geq 0$ and
$d \in \N^{*}$, and the Mat\'ern kernel of regularity $\nu>0$ and length
scale $\ell>0$. Each kernel induces a different regularity for the
functions in $\mathcal{H}_k$ \citep{ScholkopfSmola2002,SteinwartChristmann2008}.
\end{example}

Our goal is to model functions $f \colon \mathcal{X} \to \R^{d}$ -- which is
typically the case when working with randoms vectors. The multidimensional framework
goes back to \citet{MicchelliPontil2005} and \citet{CarmeliDeVitoToigo2006};
see the review of \citet{AlvarezRosascoLawrence2012} for a thorough
exposition.

\begin{definition}[Matrix-valued kernel and vector-valued RKHS]
\label{def:vRKHS}
A function $K \colon \mathcal{X}\times \mathcal{X}\to \R^{d \times d}$ is called a
\emph{matrix-valued kernel} if it is symmetric, $K(x,y) = K(y,x)^{\top}$,
and positive semi-definite in the sense that, for any $n \in \N^{*}$,
$x_{1},\dots,x_{n} \in \mathcal{X}$ and $v_{1},\dots,v_{n} \in \R^{d}$,
\[
  \sum_{i,j=1}^{n} v_{i}^{\top}\, K(x_{i},x_{j})\, v_{j} \;\geq\; 0.
\]
The associated vector-valued reproducing kernel Hilbert space is the
unique Hilbert space $\mathcal{H}_K$ of functions $f \colon \mathcal{X} \to \R^{d}$ for which
$K(\cdot, x) v \in \mathcal{H}_K$ for every $x \in\mathcal{X}$ and $v \in \R^{d}$, and
\[
  \langle f(x),\, v \rangle \;=\; \bigl\langle f,\, K(\cdot, x) v \bigr\rangle_{\mathcal{H}_K},
  \qquad \forall f \in \mathcal{H}_K,\ \forall x \in \mathcal{X},\ \forall v \in \R^{d}.
\]
\end{definition}

Two constructions are commonly encountered in practice:
diagonal kernels of the form
$K(x,y) = \mathrm{diag}\bigl(\lambda_{1} k_{1}(x,y),\,\dots,\,\lambda_{d} k_{d}(x,y)\bigr)$,
which assume independent components, and tensor kernels of the
form $K(x,y) = k(x,y)\,B$ with $k$ scalar and
$B \in \R^{d \times d}$ symmetric positive definite, which may encode
dependencies between components. Vector-valued RKHS will play
a central role in the diffusion variants of the Stein discrepancy
considered in Section~\ref{sec:background-stein}.

% ---------------------------------------------------------------------
\subsection{Kernel Mean Embeddings - KME}
\label{sec:kme-def}
% ---------------------------------------------------------------------

The KME maps a probability distribution to an element of a RKHS, in
analogy with the way the canonical feature map $\phi(x) = k(\cdot,x)$ maps
a single point to an element of $\mathcal{H}$. The construction goes back
to \citet{SmolaEtAl2007} and is reviewed in detail in \citet{MuandetEtAl2017}.

\begin{definition}[Kernel mean embedding]
\label{def:kme}
Let $k$ be a kernel on $\mathcal{X}$, $\mathcal{H}_{k}$ its associated
RKHS and $X$ a random variable on $\mathcal{X}$ with distribution
$\mathbb{P}_{X}$ such that $\E\!\left[\sqrt{k(X,X)}\right] < \infty$. The
\emph{kernel mean embedding} of $X$ into
$\mathcal{H}_{k}$ is the element $\mu_{X} \in \mathcal{H}_{k}$ defined by
\begin{equation}
  \mu_{\mathbb{P}_X}(x) \;=\; \E\bigl[k(x,X)\bigr] \;=\; \int_{\mathcal{X}} k(x,y)\,p(y)\,\mathrm{d}y,
  \qquad x \in \mathcal{X},
  \label{eq:kme-def}
\end{equation}
where $p$ denotes the density of $X$ whenever it exists. Given an independent and identically distributed (i.i.d).\
sample $X_{1},\dots,X_{n}$ from $\mathbb{P}_{X}$, an empirical estimator
of $\mu_{X}$ is
\begin{equation*}
  \widehat{\mu}_{\mathbb{P}_X} \;=\; \frac{1}{n}\sum_{i=1}^{n} k(\cdot,\,X_{i}).
\end{equation*}
\end{definition}

The integrability assumption guarantees that \eqref{eq:kme-def} defines a genuine element of $\mathcal{H}_{k}$
\citep{SmolaEtAl2007,MuandetEtAl2017}. The square root arises because the norm of the KME of $X$ has to be finite. It requires 
$\mathbb{E}\big[\|k(\cdot,X)\|_{\mathcal{H}_k}\big] < \infty$, and the 
reproducing property yields $\|k(\cdot,x)\|_{\mathcal{H}_k} = \sqrt{k(x,x)}$.

\begin{definition}[Characteristic kernel]
\label{def:characteristic}
A bounded measurable kernel $k$ on $\mathcal{X}$ is said to be
\emph{characteristic} (with respect to a class of probability measures
$\mathcal{P}$) if the map $\mathbb{P} \mapsto \mu_{\mathbb{P}}$ defined
by~\eqref{eq:kme-def} is injective on $\mathcal{P}$, i.e.\
\[
  \mu_{\mathbb{P}} \;=\; \mu_{\mathbb{Q}}
  \quad\Longrightarrow\quad
  \mathbb{P} \;=\; \mathbb{Q},
  \qquad \forall \mathbb{P}, \mathbb{Q} \in \mathcal{P}.
\]
\end{definition}
% ---------------------------------------------------------------------
\subsection{Maximum mean discrepancy}
\label{sec:mmd-def}
% ---------------------------------------------------------------------

The KME provides a natural way to quantify the dissimilarity between two
probability measures by measuring the distance, in $\mathcal{H}_{k}$,
between their embeddings. This is the idea underlying the
\emph{Maximum Mean Discrepancy} \cite{GrettonEtAl2012}.

\begin{definition}[Maximum Mean Discrepancy]
\label{def:mmd}
Let $X$ and $Y$ be random variables on $\mathcal{X}$ with distributions
$\mathbb{P}_{X}$ and $\mathbb{P}_{Y}$ respectively such that
$\mu_{\mathbb{P}_X},\,\mu_{\mathbb{P}_Y} \in \mathcal{H}_{k}$. The \emph{Maximum Mean
Discrepancy} between $\mathbb{P}_{X}$ and $\mathbb{P}_{Y}$ associated with
the kernel $k$ is the $\mathcal{H}_{k}$-norm of the difference of their
embeddings:
\begin{equation*}
  \mathrm{MMD}_{k}(\mathbb{P}_{X},\mathbb{P}_{Y})
  \;\coloneqq\;
  \bigl\| \mu_{\mathbb{P}_X} - \mu_{\mathbb{P}_Y} \bigr\|_{\mathcal{H}_{k}}.
\end{equation*}
\end{definition}

\noindent The following proposition is very useful for empirical estimation.

\begin{proposition}[Sample-form of the squared MMD]
\label{prop:mmd-sample}
Let $X,X' \overset{\textnormal{i.i.d.}}{\sim} \mathbb{P}_{X}$ and
$Y,Y' \overset{\textnormal{i.i.d.}}{\sim} \mathbb{P}_{Y}$, with $X$ and $Y$
independent. Under the integrability assumption of
Definition~\ref{def:kme}, the squared MMD admits the representation
\begin{equation}
  \mathrm{MMD}_{k}^{2}(\mathbb{P}_{X},\mathbb{P}_{Y})
  \;=\; \E\!\bigl[k(X,X')\bigr]
       + \E\!\bigl[k(Y,Y')\bigr]
       - 2\,\E\!\bigl[k(X,Y)\bigr].
  \label{eq:mmd-expectation}
\end{equation}
\end{proposition}

\begin{proof}
See \cite{GrettonEtAl2012}, Lemma 6.
\end{proof}

Expression~\eqref{eq:mmd-expectation} is fundamental for
two reasons. First, it allows the MMD to be estimated from samples
without ever computing the embeddings themselves: replacing each
expectation by its empirical counterpart yields to an unbiased
$U$-statistic estimator. Second, when the
embeddings of $\mathbb{P}_{X}$ and $\mathbb{P}_{Y}$ admit a closed-form
expression in $\mathcal{H}_{k}$, the cross-term $\E[k(X,Y)]$ can be
evaluated analytically and the MMD becomes a tractable loss for parameter
estimation. The next proposition records the discriminating property of
the MMD that motivates this loss.

\begin{proposition}[Discriminating property of the MMD]
\label{prop:mmd-discriminating}
Let $k$ be a characteristic kernel on $\mathcal{X}$ in the sense of
Definition~\ref{def:characteristic}. Then, for any probability measures
$\mathbb{P}$ and $\mathbb{Q}$ on $\mathcal{X}$ with finite KMEs,
\begin{equation*}
  \mathrm{MMD}_{k}(\mathbb{P},\mathbb{Q}) \;=\; 0
  \quad\Longleftrightarrow\quad
  \mathbb{P} \;=\; \mathbb{Q}.
\end{equation*}
\end{proposition}

% =====================================================================
\section{Background on Stein Operator and KSD}
\label{sec:background-stein}
% =====================================================================
\label{sec:stein-litterature}

This section reviews Stein's operator constructions and Kernel Stein Discrepancy (KSD). In the context of RKHS, Stein's method allows us to construct kernels adapted to specific distributions. Stein's method goes back to the seminal paper of \citet{Stein1972} and was systematized in the monograph of \citet{Stein1986}. The
\emph{density approach}, in which a Stein operator is built directly from
the score function $\nabla \log p$ of a target density $p$, has been
extended from the univariate case to a general multivariate framework by
\citet{MijouleRaicReinertSwan2023}, while \citet{GauntMijouleSwan2023}
recently provided an algebraic formalism for combining Stein operators
across independent factors. Outside the density approach,
\citet{ArrasAzmoodehPolySwan2017} derived Stein characterisations for
linear combinations of (possibly dependent) Gamma random variables
directly from a polynomial ODE satisfied by the characteristic function
-- a construction that we will use in Section~\ref{sec:stein-ggc}.

The combination of Stein's method with the kernel formalism of
Section~\ref{sec:background-rkhs-kme} is due, independently and
simultaneously, to \citet{LiuLeeJordan2016} and
\citet{ChwialkowskiStrathmannGretton2016}, who introduced what is now
called the \emph{Kernel Stein Discrepancy} (KSD) and proposed its use as
a goodness-of-fit test that does not require knowledge of the
normalising constant of $p$. Asymptotic and
non-asymptotic results for KSD-based estimators were established by
\citet{BarpEtAl2019}, with consistency, asymptotic normality
and robustness guarantees. The bulk of this
literature, however, relies on the explicit form of $\nabla \log p$,
which is precisely what is missing for distributions in the GGC class. Our approach instead derives the Stein operator directly from the characteristic function.

% ---------------------------------------------------------------------
\subsection{Stein operators}
\label{sec:stein-operator-def}
% ---------------------------------------------------------------------
A Stein operator associated to $X\sim \mathbb{P}$
 is a linear operator $\mathcal{T}_{\mathbb{P}}$ acting on a class
$\mathcal{F}$ of test functions, such that the expectation of
$\mathcal{T}_{\mathbb{P}} f(X)$ under $\mathbb{P}$ vanishes for every $f \in \mathcal{F}$.

\begin{definition}[Stein operator]
\label{def:stein-operator}
Let $\mathbb{P}$ be a probability measure on $\mathcal{X} \subseteq \R^{d}$
and let $\mathcal{F}$ be a class of (sufficiently regular) test functions
$f \colon \mathcal{X} \to \R^{d}$. A linear operator
$\mathcal{T}_{X} \colon \mathcal{F} \to L^{1}(\mathbb{P})$ is called a
\emph{Stein operator} on $\mathcal{F}$ if
\begin{equation}
  \E_{ }\!\left[\,(\mathcal{T}_{\mathbb{P}} f)(X)\,\right] \;=\; 0,
  \qquad \forall\, f \in \mathcal{F}.
  \label{eq:stein-identity}
\end{equation}
The pair $(\mathcal{T}_{\mathbb{P}}, \mathcal{F})$ is called a \emph{Stein pair} and identity~\eqref{eq:stein-identity} is the
associated \emph{Stein identity}.
\end{definition}

The class $\mathcal{F}$ is generally chosen so that the Stein
identity~\eqref{eq:stein-identity} characterises $\mathbb{P}$, in the
sense that $Y \overset{d}{=} X$ if and only if
$\E[(\mathcal{T}_{\mathbb{P}} f)(Y)] = 0$ for every $f \in \mathcal{F}$.

\begin{proposition}[Multivariate density-based Stein operator]
\label{prop:stein-density-multi}
Let $\mathbb{P}$ admit a strictly positive, differentiable density $p$
on $\R^{d}$ and let $\mathcal{F}$ the set of admissible tests functions. Then the \emph{Langevin--Stein operator}
\begin{equation*}
  (\mathcal{T}_{\mathbb{P}} f)(x)
  \;=\; \mathrm{div}(f(x)) \,+\, \ip{\nabla \log p(x)}{f(x)},
  \qquad x \in \R^{d},
\end{equation*}
is a Stein operator on $\mathcal{F}$, where
$\mathrm{div}(f(x)) = \sum_{i=1}^{d} \partial f_{i}/\partial x_{i}$ is the
divergence of $f$ and with:
\begin{equation*}
  \mathcal{F} \;=\; \Bigl\{ f \in C^{1}(\R^{d}; \R^{d}) \,:\,
       \lim_{\|x\| \to \infty} \|f(x)\|\,p(x) = 0 \Bigr\}.
\end{equation*}
with $p$ the density of the r.v. X.
\end{proposition}

\begin{proof}
    See \cite{MijouleRaicReinertSwan2023}, Propostion 3.14.
\end{proof}

\begin{remark}[Limitation for the GGC class]
\label{rem:stein-limit-ggc}
Propositions ~\ref{prop:stein-density-multi} relies on the score function
$\nabla \log p$. For a distribution in $\mathcal{G}_{d}(\nu)$, however, the
density $p$ is in general not explicit. The construction will therefore have to be carried out
from a different starting point: the partial differential equation
satisfied by the characteristic function $\charf_{X}$ of $X \sim \mathcal{G}_{d}(\nu)$,
following the strategy initiated by~\citet{ArrasAzmoodehPolySwan2017} for
linear combinations of Gamma distributions. This is the subject of
Section~\ref{sec:stein-ggc}.
\end{remark}

% ---------------------------------------------------------------------
\subsection{Kernel Stein Discrepancy}
\label{sec:ksd-def}
% ---------------------------------------------------------------------

We now combine the Stein operator $\mathcal{T}_{\mathbb{P}}$ of
Section~\ref{sec:stein-operator-def} with the RKHS framework of
Section~\ref{sec:background-rkhs-kme}. In fact, Stein operators allow to construct kernels adapted to $\mathbb{P}$.

\begin{definition}[Kernel Stein Discrepancy \cite{LiuLeeJordan2016}]
\label{def:ksd}
Let $X\sim\mathbb{P}$ be a r.v on $\mathcal{X} \subseteq \R^{d}$
admitting a Stein pair $(\mathcal{T}_{\mathbb{P}}, \mathcal{F})$, and let
$\mathcal{H}_{K}$ be a vector-valued RKHS of functions
$f \colon \mathcal{X} \to \R^{d}$ with matrix-valued reproducing kernel
$K$, such that $\mathcal{H}_{K} \cap \mathcal{F}$ is dense in
$\mathcal{F}$. The \emph{Kernel Stein Discrepancy} from a r.v $Y\sim\mathbb{Q}$ on $\mathcal{X}$ is
\begin{equation}
  \mathrm{KSD}(\mathbb{P}, \mathbb{Q})
  \;\coloneqq\;
  \sup_{f \in \mathcal{H}_{K},\,\|f\|_{\mathcal{H}_{K}} \leq 1}
       \E\!\left[\,(\mathcal{T}_{\mathbb{P}} f)(Y)\,\right].
  \label{eq:ksd-sup}
\end{equation}
\end{definition}

\begin{definition}[Stein kernel]
\label{def:stein-kernel}
Let $K$ be a (matrix-valued) reproducing kernel on $\mathcal{X}$ and
$\mathcal{T}_{\mathbb{P}}$ a Stein operator for $X$. The \emph{Stein
kernel} associated to $X$ and $K$ is
\begin{equation*}
  K_{\mathbb{P}}(x,y)
  \;\coloneqq\;
  \mathcal{T}_{\mathbb{P}}^{(x)}\, \mathcal{T}_{\mathbb{P}}^{(y)}\, K(x,y),
  \qquad x, y \in \mathcal{X},
\end{equation*}
where $\mathcal{T}_{\mathbb{P}}^{(x)}$ (resp.\ $\mathcal{T}_{\mathbb{P}}^{(y)}$) denotes the
action of $\mathcal{T}_{X}$ on the function
$x \mapsto K(x,y)$ (resp.\ $y \mapsto K(x,y)$).
\end{definition}

\noindent \eqref{eq:stein-identity} leads to very simple form for the KSD for Stein's kernels.

\begin{proposition}[]
\label{prop:ksd-norm}
Under the assumptions of Definitions~\ref{def:ksd}
and~\ref{def:stein-kernel}, the Stein kernel $K_{\mathbb{P}}$ is positive
semi-definite and the kernel mean embedding of $X$ in
$\mathcal{H}_{K_{\mathbb{P}}}$ is:
\begin{equation}
  \mu_{\mathbb{P}}^{K_{\mathbb{P}}}
  \;\coloneqq\; \E \bigl[K_{\mathbb{P}}(\cdot, X)\bigr]
  \;=\; 0.
  \label{eq:kme-zero}
\end{equation}
Consequently, KSD verifing:
\begin{equation}
  \mathrm{KSD}^{2}(\mathbb{P}, \mathbb{Q})
  \;=\; \bigl\| \mu_{\mathbb{Q}}^{K_\mathbb{P}} \bigr\|_{\mathcal{H}_{K_{\mathbb{P}}}}^{2}
  \;=\; \mathrm{MMD}_{K_{\mathbb{P}}}^{2}(\mathbb{P}, \mathbb{Q})
  \;=\; \E\bigl[K_{\mathbb{P}}(Y, Y')\bigr],
  \label{eq:ksd-norm}
\end{equation}
where $Y, Y' \overset{\textnormal{i.i.d.}}{\sim} \mathbb{Q}$ and
$\mu_{\mathbb{Q}}^{K_{\mathbb{P}}} = \E[K_{X}(\cdot, Y)]$
is the KME of $Y$ in $\mathcal{H}_{K_{\mathbb{P}}}$.
\end{proposition}

\begin{proof}
Identity~\eqref{eq:kme-zero} follows from the Stein
identity~\eqref{eq:stein-identity} applied to the function
$x \mapsto \mathcal{T}_{\mathbb{P}}^{(y)} K(x,y) \in \mathcal{H}_{K}$, which lies in
$\mathcal{F}$ for each fixed $y$. Indeed,
\[
  \mu_{\mathbb{P}}^{K_{\mathbb{P}}}(y)
  \;=\; \E\bigl[\mathcal{T}_{\mathbb{P}}^{(x)}\, \mathcal{T}_{\mathbb{P}}^{(y)}\, K(X,y)\bigr]
  \;=\; \mathcal{T}_{\mathbb{P}}^{(y)}\,\E\bigl[\mathcal{T}_{\mathbb{P}}^{(x)}\, K(X,y)\bigr]
  \;=\; \mathcal{T}_{\mathbb{P}}^{(y)}\, 0 \;=\; 0.
\]
By the variational characterisation of the RKHS norm and the
reproducing property of $\mathcal{H}_{K_{\mathbb{P}}}$, the
supremum~\eqref{eq:ksd-sup} equals
$\| \mu_{\mathbb{Q}}^{K_{\mathbb{P}}} \|_{\mathcal{H}_{K_{\mathbb{P}}}}$,
which gives the first equality in~\eqref{eq:ksd-norm}. Plugging
$\mu_{\mathbb{P}}^{K_{\mathbb{P}}} = 0$ into
Proposition~\ref{prop:mmd-sample} gives the result. For more details about this result, see \cite{ChwialkowskiStrathmannGretton2016}, Section 2.1.
\end{proof}

\subsubsection*{Sample-form: $U$-statistics}

Given an i.i.d.\ sample $Y_{1}, \dots, Y_{n} \sim \mathbb{Q}$, the
expectation~\eqref{eq:ksd-norm} can be approximated by either of the 
empirical estimators in \citet{LiuLeeJordan2016} and
\citet{ChwialkowskiStrathmannGretton2016}.

\begin{definition}[]
\label{def:ksd-empirical}

The empirical estimator of $\mathrm{KSD}^{2}(\mathbb{P},\mathbb{Q})$ is the U-statistic:
\begin{align*}
  \widehat{\mathrm{KSD}}_{U}^{2}
  &\;=\; \frac{1}{n(n-1)} \sum_{1 \leq i \neq j \leq n}
          K_{\mathbb{P}}(Y_{i}, Y_{j})
\end{align*}
\end{definition}

\begin{proposition}[Convergence of the empirical KSD]
\label{prop:ksd-consistency}
Suppose $K_{\mathbb{P}}$ is measurable and integrable, in the sense that
$\E[|K_{\mathbb{P}}(Y,Y')|] < \infty$ for $Y, Y' \overset{\textnormal{i.i.d.}}{\sim} \mathbb{Q}$.
Then:
\begin{equation*}
  \widehat{\mathrm{KSD}}_{U}^{2}
  \;\xrightarrow[n \to \infty]{\textnormal{a.s.}}\;
  \mathrm{KSD}^{2}(\mathbb{P},\mathbb{Q})
\end{equation*}
Moreover, $\widehat{\mathrm{KSD}}_{U}^{2}$ is unbiased.
\end{proposition}

The proof of Proposition~\ref{prop:ksd-consistency} is a direct
application of the strong law of large numbers for $U$-statistics; see~\citep{LiuLeeJordan2016,BarpEtAl2019} for details and
for the corresponding central limit theorems.

Let $\{\mathbb{P}_{\theta}\}_{\theta \in \Theta}$ be a parametric family
of target distributions and let $Y_{1}, \dots, Y_{n}$ be i.i.d.\ from an
unknown distribution $\mathbb{Q}$. Proposition~\ref{prop:ksd-consistency} suggests the
following estimation principle, called \emph{Minimum Stein Discrepancy
Estimator} (MSDE) by~\citet{BarpEtAl2019}:
\begin{equation}
  \widehat{\theta}_{n}
  \;\in\; \underset{\theta \in \Theta}{\mathrm{argmin}}\,
          \widehat{\mathrm{KSD}}^{2}\bigl(\mathbb{P}_\theta, \hat{\mathbb{Q}}_n\bigr),
  \label{eq:msde}
\end{equation}
where $X\sim \mathbb{P}_\theta$ and, $\widehat{\mathbb{Q}}_{n} = \frac{1}{n}\sum_{i=1}^{n} \delta_{Y_{i}}$
is the empirical distribution of the sample.
% =====================================================================
\section{Stein operators for GGC}
\label{sec:stein-ggc}
% =====================================================================
 
The aim of this section is to construct an explicit Stein operator for
the multivariate GGC class $\mathcal{G}_{d}(\nu)$. Since the density of a
GGC is in general not explicit (Remark~\ref{rem:stein-limit-ggc}), the density approach of
Proposition \ref{prop:stein-density-multi}
is not applicable. We instead provide \emph{two distinct constructions}
of a Stein operator for the GGC class, both starting from the
characteristic function rather than from the density.
 
The first construction (Section~\ref{sec:arras-multi-lemma}) is a
multivariate extension of Lemma 2.1 from \cite{ArrasAzmoodehPolySwan2017} and applies to any
random vector whose characteristic function satisfies a polynomial
partial differential equation.
 
The second construction (Section~\ref{sec:stein-ggc-integral}) is
tailored specifically to a $\mathcal{G}_{d}(\nu)$ distribution and
exploits the integral form of its log-characteristic function. The
resulting operator has integral coefficients that depend only on the
Thorin measure $\nu$.
 
We next interpret the second operator in terms of the L\'evy measure of
$X$ and of the Gamma subordinators that compose the GGC
(Section~\ref{sec:levy-subordinators}). Finally, we show that on the
discrete sub-class $\mathcal{G}_{d,n}$, where the Thorin measure has
finitely many atoms, the two operators are equal
(Section~\ref{sec:gdn-equivalence}).
 
% ---------------------------------------------------------------------
\subsection{A multivariate Stein lemma via the characteristic function}
\label{sec:arras-multi-lemma}
% ---------------------------------------------------------------------
 
The univariate result of \citet[Lemma~2.1]{ArrasAzmoodehPolySwan2017}
identifies a Stein operator with polynomial coefficients for any random
variable whose characteristic function satisfies a first-order ODE with
polynomial coefficients. The following lemma is its multivariate
counterpart.
 \begin{proposition}[Multivariate Stein lemma via the characteristic function]
\label{lem:arras-multi}
Let $X\sim\mathbb{P}$ be a random vector with values in $\R^{d}$ and characteristic
function $\charf_{X}(t) = \E[e^{\im\,\ip{t}{X}}]$, $t \in \R^{d}$ with $f\in \mathbb{S}(\mathbb{R}^d, \mathbb{R}^d)$ where $\mathbb{S}(\mathbb{R}^d, \mathbb{R}^d)$ is the Schwartz class of functions. Let us denote $\widetilde{f}$, the Fourier transform of $f$ with $f: \mathbb{R}^d\to \mathbb{R}^d $.
Suppose there exist a polynomial $A \colon \R^{d} \to \R$ and a
polynomial vector field $B \colon \R^{d} \to \R^{d}$ such that
$\charf_{X}$ satisfies the partial differential equation (PDE)
\begin{equation}
  A(t)\,\nabla_{t}\charf_{X}(t)
  \;=\; \im\, B(t)\,\charf_{X}(t),
  \qquad t \in \R^{d}.
  \label{eq:cf-pde}
\end{equation}
For a multi-index $\beta = (\beta_{1},\dots,\beta_{d}) \in \N^{d}$,
denote
$D^{\beta} = \partial^{|\beta|}/(\partial x_{1}^{\beta_{1}}\cdots \partial x_{d}^{\beta_{d}})$,
and, for $A(t) = \sum_{\beta} a_{\beta}\, t^{\beta}$ and
$B_{\ell}(t) = \sum_{\beta} b_{\ell,\beta}\, t^{\beta}$ ($\ell = 1, \dots, d$),
consider the differential operators
\[
  A(D) f \;=\; \sum_{\beta} a_{\beta}\, D^{\beta} f,
  \qquad
  B_{\ell}(D) f \;=\; \sum_{\beta} b_{\ell,\beta}\, D^{\beta} f,
\]
which satisfy
$\widetilde{A(D) f_{\ell}}(u) = A(\im\, u)\,\widetilde{f_{\ell}}(u)$ and
$\widetilde{B_{\ell}(D) f_{\ell}}(u) = B_{\ell}(\im\, u)\,\widetilde{f_{\ell}}(u)$
for every $\ell = 1,\dots,d$. Then the operator
\begin{equation}
  (\mathcal{T}_{\mathbb{P}} f)(x)
  \;=\; \ip{x}{A(D)\,f(x)} \,-\, \ip{B(D)}{f(x)},
  \qquad x \in \R^{d},
  \label{eq:stein-arras-multi}
\end{equation}
is a Stein operator for $X$ on the multivariate Schwartz class
$\mathcal{F} = \Sclass(\R^{d};\R^{d})$, in the sense that
$\E[(\mathcal{T}_{\mathbb{P}} f)(X)] = 0$ for every $f \in \mathcal{F}$.
\end{proposition}

\smallskip

\begin{proof}

We follow the strategy of \citet{ArrasAzmoodehPolySwan2017} adapted to
the multivariate setting. Let $f \in \Sclass(\R^{d}, \R^{d})$ and write
$f = (f_{1},\dots,f_{d})^{\top}$. The Fourier transforms on
$\Sclass(\R^{d})$ is given by
\[
  \widetilde{f_{\ell}}(u) \;=\; (2\pi)^{-d/2}\!\int_{\R^{d}}\! f_{\ell}(x)\, e^{\im\,\ip{u}{x}}\,\mathrm{d}x,
  \qquad u \in \R^{d},
\]
Also, we have 
$\widetilde{A(D) f_{k}}(u) = A(\im\, u)\,\widetilde{f_{k}}(u)$, summing over
$k = 1,\dots,d$ gives
\[
  \E\bigl[\ip{X}{A(D) f(X)}\bigr]
  \;=\; \frac{\im}{(2\pi)^{d/2}}\!\int_{\R^{d}}\! A(\im\, u)\,\ip{\widetilde{f}(u)}{\nabla_{u}\charf_{X}(-u)}\,\mathrm{d}u.
\]
 
\smallskip\noindent Similarly, for $B$,
\[
  \E\bigl[\ip{B(D)}{f(X)}\bigr]
  \;=\; (2\pi)^{-d/2}\!\int_{\R^{d}}\!\ip{B(\im\, u)}{\widetilde{f}(u)}\,\charf_{X}(-u)\,\mathrm{d}u.
\]
 
\smallskip\noindent The PDE~\eqref{eq:cf-pde} evaluated at $-u$ writes
$A(-u)\,\nabla_{u}\charf_{X}(-u) = -\im\, B(-u)\,\charf_{X}(-u)$, or
equivalently, after substituting $u \to \im\, u$ in the polynomial
coefficients,
$A(\im\, u)\,\nabla_{u}\charf_{X}(-u) = -\im\, B(\im\, u)\,\charf_{X}(-u)$.
Plugging this in the integrand,
\[
  A(\im\, u)\,\ip{\widetilde{f}(u)}{\nabla_{u}\charf_{X}(-u)}
  \;=\; -\im\,\ip{\widetilde{f}(u)}{B(\im\, u)\,\charf_{X}(-u)}
  \;=\; -\im\,\ip{B(\im\, u)}{\widetilde{f}(u)}\,\charf_{X}(-u).
\]
Multiplying by $\im / (2\pi)^{d/2}$ and using $\im\cdot(-\im) = 1$,
\[
  \E\bigl[\ip{X}{A(D) f(X)}\bigr]
  \;=\; \E\bigl[\ip{B(D)}{f(X)}\bigr],
\]
which is exactly $\E[(\mathcal{T}_{\mathbb{P}} f)(X)] = 0$ by definition of $\mathcal{T}_{\mathbb{P}}$
in~\eqref{eq:stein-arras-multi}. And thus $\mathcal{T}_\mathbb{P}$ is a Stein operator.
\end{proof}
% ---------------------------------------------------------------------
\subsection{ A Stein operator for a generalized gamma convolution}
\label{sec:stein-ggc-integral}
% ---------------------------------------------------------------------
 
We now construct a second Stein operator, tailored
for the structure of $\mathcal{G}_{d}(\nu)$. The starting
point is again the partial differential equation satisfied by
$\charf_{X}$, but instead of seeking polynomial coefficients $A$ and
$B$ as in Proposition~\ref{lem:arras-multi}, we exploit directly the
integral representation of $\log\charf_{X}$ in
Definition~\ref{def:Gd-classes}. The result is a clean integral form
of the Stein operator whose coefficients depend only on the Thorin
measure $\nu$, and which is valid on a (typically larger) class of test
functions.
 
\begin{theorem}[Stein operator for $\mathcal{G}_{d}(\nu)$]
\label{lem:stein-ggc}
Let $\nu$ be a Thorin measure on $\Rp^{d}$ in the sense of
Definition~\ref{def:thorin-measure}, satisfying the additional
integrability condition
\begin{equation}
  \int_{\Rp^{d}} \|s\|\,\nu(\mathrm{d}s) \;<\; \infty,
  \label{eq:nu-integrability}
\end{equation}
and let $X \sim \mathcal{G}_{d}(\nu)$. Let $\mathcal{F}_{\nu}$ denote a
class of test functions $f \colon \R^{d} \to \R^{d}$ for which the Stein identity below is well-defined. Then the operator
\begin{equation}
  (\mathcal{T}_{\nu} f)(x)
  \;=\; \ip{x}{f(x)}
       \,-\, \int_{\Rp^{d}}\!\int_{0}^{\infty}\! e^{-z}\,\ip{s}{f(x + z s)}
              \,\mathrm{d}z\,\nu(\mathrm{d}s),
  \qquad x \in \R^{d},
  \label{eq:stein-ggc-integral}
\end{equation}
\begin{equation}
  \;=\; \ip{x}{f(x)} \,-\, \int_{\Rp^{d}}\!\E\bigl[\,\ip{s}{f(x + s\, E)}\,\bigr]\,\nu(\mathrm{d}s),\qquad E\sim \mathrm{exp(1)}
  \label{eq:stein-ggc-subordinator}
\end{equation}
is a Stein operator for $X$ on $\mathcal{F}_{X}$. Equivalently, for
every $f \in \mathcal{F}_{X}$,
\[
  \E\bigl[\,\ip{X}{f(X)}\,\bigr]
  \;=\; \int_{\Rp^{d}}\!\int_{0}^{\infty}\! e^{-z}\,
        \E\bigl[\,\ip{s}{f(X + z s)}\,\bigr]\,\mathrm{d}z\,\nu(\mathrm{d}s).
\]

\end{theorem}

\smallskip
 
\begin{proof}
 
\smallskip\noindent
By Definition~\ref{def:Gd-classes} and the chain rule applied to
$\log\charf_{X}$,
\[
  \nabla_{t}\log\charf_{X}(t)
  \;=\; \im\!\int_{\Rp^{d}}\!\frac{s}{1 - \im\,\ip{s}{t}}\,\nu(\mathrm{d}s),
\]
hence
\[
  \nabla_{t}\charf_{X}(t)
  \;=\; \im\,\charf_{X}(t)\!\int_{\Rp^{d}}\!\frac{s}{1 - \im\,\ip{s}{t}}\,\nu(\mathrm{d}s).
\]
Using the integral representation
$\displaystyle \frac{1}{1 - \im\,\ip{s}{t}} = \int_{0}^{\infty}\! e^{-z(1 - \im\,\ip{s}{t})}\,\mathrm{d}z$
(valid since $\Re(1 - \im\,\ip{s}{t}) = 1 > 0$),
\begin{equation}
  \tfrac{1}{\im}\,\nabla_{t}\charf_{X}(t)
  \;=\; \int_{\Rp^{d}}\!\int_{0}^{\infty}\! s\, e^{-z(1 - \im\,\ip{s}{t})}\,\charf_{X}(t)\,\mathrm{d}z\,\nu(\mathrm{d}s).
  \label{eq:proof-cf-derivative}
\end{equation}
 
\smallskip\noindent Since
$\charf_{X}(t)\, e^{\im\, z\,\ip{s}{t}} = \E\!\bigl[e^{\im\,\ip{X + z s}{t}}\bigr]$
and 
$\tfrac{1}{\im}\,\nabla_{t}\charf_{X}(t) = \E\!\bigl[X\, e^{\im\,\ip{X}{t}}\bigr]$,
\eqref{eq:proof-cf-derivative} writes
\begin{equation}
  \E\!\bigl[X\, e^{\im\,\ip{X}{t}}\bigr]
  \;=\; \int_{\Rp^{d}}\!\int_{0}^{\infty}\! s\, e^{-z}\,\E\!\bigl[e^{\im\,\ip{X + z s}{t}}\bigr]\,\mathrm{d}z\,\nu(\mathrm{d}s).
  \label{eq:proof-cf-prob}
\end{equation}
 
\smallskip\noindent Write for the left-hand side of \eqref{eq:proof-cf-prob}:
\[
  (2\pi)^{d/2}\!\int_{\R^{d}}\!\bigl\langle\E\!\bigl[X\, e^{\im\,\ip{X}{t}}\bigr],\widetilde{f}(t)\bigr\rangle\,\mathrm{d}t
  \;=\; \E\!\left[\,\Bigl\langle X,\,(2\pi)^{d/2}\!\int_{\R^{d}}\!\widetilde{f}(t)\, e^{\im\,\ip{X}{t}}\,\mathrm{d}t\,\Bigr\rangle\,\right]
  \;=\; \E\!\bigl[\ip{X}{f(X)}\bigr],
\]
For the right-hand side of \eqref{eq:proof-cf-prob}:
\[
  \int_{\Rp^{d}}\!\int_{0}^{\infty}\! e^{-z}\,
  \E\!\left[\,\Bigl\langle s,\,(2\pi)^{d/2}\!\int_{\R^{d}}\!\widetilde{f}(t)\, e^{\im\,\ip{X + z s}{t}}\,\mathrm{d}t\,\Bigr\rangle\,\right]
  \mathrm{d}z\,\nu(\mathrm{d}s)
  \;=\; \int_{\Rp^{d}}\!\int_{0}^{\infty}\! e^{-z}\,\E\!\bigl[\ip{s}{f(X + z s)}\bigr]\,\mathrm{d}z\,\nu(\mathrm{d}s),
\]

\smallskip\noindent Combining the two sides yields the announced Stein identity
\[
  \E\!\bigl[\ip{X}{f(X)}\bigr]
  \;=\; \int_{\Rp^{d}}\!\int_{0}^{\infty}\! e^{-z}\,\E\!\bigl[\ip{s}{f(X + z s)}\bigr]\,\mathrm{d}z\,\nu(\mathrm{d}s),
\]
i.e.\ $\E[(\mathcal{T}_{\nu} f)(X)] = 0$ for every $f \in \mathcal{F}_{\nu}$ and thus, $\mathcal{T}_\nu$ is a Stein operator.
\end{proof}

Finally remark that~\eqref{eq:stein-ggc-integral} has a probabilistic interpretation that highlights the role of Gamma
subordinators in the construction. Since $z \mapsto e^{-z}$ is the
density of an exponential random variable $E \sim \mathrm{Exp}(1)$,
\[
  \int_{0}^{\infty}\! e^{-z}\,\ip{s}{f(x + z s)}\,\mathrm{d}z
  \;=\; \E\bigl[\,\ip{s}{f(x + s\, E)}\,\bigr],
\]
which gives equation \eqref{eq:stein-ggc-subordinator}.

\begin{remark}[Admissible test functions and moment condition]
The Stein identity of Theorem~\ref{lem:stein-ggc} holds, in particular, for every bounded Borel-measurable function
$f:\mathbb{R}^{d}\to\mathbb{R}^{d}$. Moreover, the condition
$\displaystyle \int_{\mathbb{R}_{+}^{d}}\|s\|\,\nu(\mathrm{d}s)<\infty$ is equivalent to the existence of a finite first moment for the associated MGGC distribution.
\end{remark}

% ---------------------------------------------------------------------
\subsection{L\'evy measure interpretation}
\label{sec:levy-subordinators}
% ---------------------------------------------------------------------
 
A multivariate generalised Gamma convolution is infinitely divisible
(indeed, even self-decomposable, see~\cite{Bondesson1992}) and admits
a L\'evy--Khintchine representation as a pure-jump subordinator  with
values in $\Rp^{d}$ \cite{BarndorffNielsenMaejimaSato2006}. We now make this representation explicit and show
that the equation~\eqref{eq:stein-ggc-integral} can be rewritten directly in terms of the L\'evy measure.
  
\begin{proposition}[L\'evy measure of $\mathcal{G}_{d}(\nu)$]
\label{prop:levy-measure-multi}
Let $\nu$ be a Thorin measure on $\Rp^{d}$ and $X \sim \mathcal{G}_{d}(\nu)$.
Define the positive Radon measure $\Pi$ on $\Rp^{d}\setminus\{0\}$ as
the pushforward of $\nu \otimes \frac{e^{-z}}{z}\,\mathrm{d}z$ under the
map $(s,z) \mapsto z s$, that is,
\begin{equation}
  \int_{\Rp^{d}\setminus\{0\}}\! g(y)\,\Pi(\mathrm{d}y)
  \;\coloneqq\;
  \int_{\Rp^{d}}\!\int_{0}^{\infty}\! g(z s)\,\frac{e^{-z}}{z}\,\mathrm{d}z\,\nu(\mathrm{d}s)
  \label{eq:levy-multi-def}
\end{equation}
for every Borel-measurable $g \colon \Rp^{d}\setminus\{0\} \to \R_{+}$.
Then $\Pi$ is the L\'evy measure of $X$, in the sense that
\begin{equation}
  \log \charf_{X}(t)
  \;=\; \int_{\Rp^{d}\setminus\{0\}}\!\!\bigl(e^{\im\,\ip{t}{y}} - 1\bigr)\,\Pi(\mathrm{d}y),
  \qquad t \in \R^{d}.
  \label{eq:LK-multi}
\end{equation}
The local integrability of $\Pi$ follows from the integration conditions
imposed on the Thorin measure $\nu$ \cite{Bondesson1992}:
$$\int_{\mathbb{R}_+^d}\bigl(1\wedge\lVert x\rVert\bigr)\,
\Pi(\mathrm{d}x)<\infty.$$
\end{proposition}

\smallskip
 
\begin{proof}
 
The classical Frullani-type identity states that, for any complex
number $w$ with $\Re(w) > 0$,
\[
  -\log w \;=\; \int_{0}^{\infty}\!\bigl(e^{-z w} - e^{-z}\bigr)\,\frac{\mathrm{d}z}{z}.
\]
Applied to $w = 1 - \im\,\ip{s}{t}$ (which has positive real part),
\[
  -\log(1 - \im\,\ip{s}{t})
  \;=\; \int_{0}^{\infty}\!\bigl(e^{\im\, z\,\ip{s}{t}} - 1\bigr)\,\frac{e^{-z}}{z}\,\mathrm{d}z.
\]
Substituting in~\eqref{eq:cf-Gd}, exchanging the integrals (justified by
the integrability of $\nu$ and Fubini--Tonelli) and using the
pushforward definition~\eqref{eq:levy-multi-def} as a change of
variable,
\begin{align*}
  \log\charf_{X}(t)
  &\;=\; \int_{\Rp^{d}}\!\int_{0}^{\infty}\!\bigl(e^{\im\, z\,\ip{s}{t}} - 1\bigr)\,\frac{e^{-z}}{z}\,\mathrm{d}z\,\nu(\mathrm{d}s) \\
  &\;=\; \int_{\Rp^{d}\setminus\{0\}}\!\!\bigl(e^{\im\,\ip{t}{y}} - 1\bigr)\,\Pi(\mathrm{d}y),
\end{align*}
which is~\eqref{eq:LK-multi} and identifies $\Pi$ as the L\'evy measure
of $X$. 
\end{proof}

\begin{theorem}[L\'evy form of the Stein operator]
\label{prop:stein-levy-form}
Under the assumptions of Theorem~\ref{lem:stein-ggc}, the Stein
operator~\eqref{eq:stein-ggc-integral} can be rewritten as
\begin{equation}
  (\mathcal{T}_{\nu} f)(x)
  \;=\; \ip{x}{f(x)}
       \,-\, \int_{\Rp^{d}\setminus\{0\}}\!\!\ip{y}{f(x + y)}\,\Pi(\mathrm{d}y),
  \qquad x \in \R^{d},
  \label{eq:stein-ggc-levy}
\end{equation}
\end{theorem}
where $\Pi$ is the Lévy measure defined by \eqref{eq:levy-multi-def}

\smallskip

\begin{proof}
    
Let $z > 0$, $\ip{s}{f(x + z s)} = \frac{1}{z}\,\ip{z s}{f(x + z s)}$, hence
\[
  e^{-z}\,\ip{s}{f(x + z s)}
  \;=\; \frac{e^{-z}}{z}\,\ip{z s}{f(x + z s)}.
\]
Therefore,
\begin{align*}
  \int_{\Rp^{d}}\!\int_{0}^{\infty}\! e^{-z}\,\ip{s}{f(x + z s)}\,\mathrm{d}z\,\nu(\mathrm{d}s)
  &\;=\; \int_{\Rp^{d}}\!\int_{0}^{\infty}\!\ip{z s}{f(x + z s)}\,\frac{e^{-z}}{z}\,\mathrm{d}z\,\nu(\mathrm{d}s) \\
  &\;=\; \int_{\Rp^{d}\setminus\{0\}}\!\ip{y}{f(x + y)}\,\Pi(\mathrm{d}y),
\end{align*}
the last equality being~\eqref{eq:levy-multi-def} applied to the bounded Borel-measurable function
$g(y) = \ip{y}{f(x + y)}$. Substituting in~\eqref{eq:stein-ggc-integral}
yields~\eqref{eq:stein-ggc-levy}. 
\end{proof}

% ---------------------------------------------------------------------
\subsection{Discrete case: equivalence of the two forms}
\label{sec:gdn-equivalence}
% ---------------------------------------------------------------------
 
For a finite gamma convolution $X \sim \mathcal{G}_{d,n}(\alpha, s)$,
the Thorin measure is purely atomic, $\nu = \sum_{j=1}^{n} \alpha_{j}\,\delta_{s_{j}}$,
and the characteristic function is the product
$\charf_{X}(t) = \prod_{j=1}^{n}(1 - \im\,\ip{s_{j}}{t})^{-\alpha_{j}}$.
Both constructions of the previous sections are then well-defined on
the Schwartz class, and we now show that they yield the same
operator.
 
\paragraph{Construction $1$ (multivariate Arras et al.\ form).}
Multiplying $\nabla_{t}\charf_{X}(t)$ by $\prod_{j=1}^{n}(1 - \im\,\ip{s_{j}}{t})$
turns the partial differential equation satisfied by $\charf_{X}$ into
the polynomial form~\eqref{eq:cf-pde} with
\begin{equation}
  A(t) \;=\; \prod_{j=1}^{n} (1 - \im\,\ip{s_{j}}{t}),
  \qquad
  B(t) \;=\; \sum_{j=1}^{n} \alpha_{j}\, s_{j}\,\prod_{k \neq j}(1 - \im\,\ip{s_{k}}{t}).
  \label{eq:AB-discrete}
\end{equation}
By Proposition~\ref{lem:arras-multi} the Stein operator is:
\begin{align}
  (\mathcal{T}^{\mathrm{Arras}}_{} f)(x)
  &\;=\; \Bigl\langle x \,-\, {\textstyle\sum_{j=1}^{n}}\,\alpha_{j}\, s_{j},\, f(x)\Bigr\rangle \notag\\
  &\quad +\, \sum_{\ell=1}^{n-1}(-1)^{\ell}\!\left[
       \Bigl\langle x - {\textstyle\sum_{j=1}^{n}}\,\alpha_{j}\, s_{j},\, e_{\ell}(D^{*}) f(x)\Bigr\rangle
       - \sum_{j=1}^{n}\!\bigl\langle\alpha_{j}\, s_{j},\, (e_{\ell}(D^{*}) - e_{\ell}((D^{*})_{j}))f(x)\bigr\rangle
   \right]\notag\\
  &\quad +\, (-1)^{n}\!\Bigl\langle x,\, {\textstyle\prod_{j=1}^{n}}\, D_{j}^{*}\, f(x)\Bigr\rangle,
  \label{eq:stein-arras-discrete}
\end{align}
where $D_{j}^{*} = \ip{s_{j}}{\nabla_{x}}$ is the directional derivative
in direction $s_{j}$, $e_{\ell}$ is the elementary symmetric polynomial
of degree $\ell$, and $e_{\ell}((D^{*})_{j})$ denotes the same polynomial
evaluated at $(D_{1}^{*},\dots,D_{n}^{*})$ with the $j$-th entry
omitted.
 
\paragraph{Construction $2$ (GGC integral form).}
Applying Theorem~\ref{lem:stein-ggc} to the atomic Thorin measure
$\nu = \sum_{j=1}^{n} \alpha_{j}\,\delta_{s_{j}}$
gives the Stein operator
\begin{equation}
  (\mathcal{T}^{\mathrm{GGC}}_{} f)(x)
  \;=\; \ip{x}{f(x)} \,-\, \sum_{j=1}^{n}\alpha_{j}\!\int_{0}^{\infty}\! e^{-z}\,\ip{s_{j}}{f(x + z s_{j})}\,\mathrm{d}z.
  \label{eq:stein-ggc-discrete}
\end{equation}
 
The following proposition shows that these two operators coincide.
 
\begin{proposition}[Equivalence of the Arras et al.\ and GGC operators on $\mathcal{G}_{d,n}$]
\label{prop:equivalence-gdn}
For every $X \sim \mathcal{G}_{d,n}(\alpha, s)$ and every
$f \in \Sclass(\R^{d}; \R^{d})$,
$$
  (\mathcal{T}^{\mathrm{Arras}}_{} f)(x)
  \;=\; (\mathcal{T}^{\mathrm{GGC}}_{} f)(x),
  \qquad x \in \R^{d}.
$$

\end{proposition}
 
\begin{proof}
See Appendix~\ref{app:proof-equivalence-gdn}.
\end{proof}
 
\begin{remark}[Why the equivalence matters]
\label{rem:equivalence-meaning}
Proposition~\ref{prop:equivalence-gdn} establishes that, on the
discrete sub-class $\mathcal{G}_{d,n}$, the two independent
constructions of Sections~\ref{sec:arras-multi-lemma}
and~\ref{sec:stein-ggc-integral} produce the same differential
expression. As $\nu$ ranges from finitely atomic to absolutely
continuous measures, the formal differential
operator~\eqref{eq:stein-arras-discrete} -- of unbounded order in $n$
and generally ill-defined as an actual differential operator when
$\nu$ is diffuse -- ceases to make sense, while the integral
operator~\eqref{eq:stein-ggc-integral} remains well-defined. From the
perspective of parameter estimation in Section~\ref{sec:ksd-def}, this
means that the KSD built on $\mathcal{T}^{\mathrm{GGC}}_{}$ remains
computable throughout the closure $\mathcal{G}_{d}$ of
$\bigcup_{n \geq 1}\mathcal{G}_{d,n}$, whereas the
operator~\eqref{eq:stein-arras-discrete} only makes sense on the
discrete sub-class $\mathcal{G}_{d,n}$.
\end{remark}

% =====================================================================
\section{Kernel Stein Discrepancy for MGGC}
\label{sec:ksd-ggc}
% =====================================================================

We are now in a position to combine the Stein operator of
Section~\ref{sec:stein-ggc} with the kernel framework of
Section~\ref{sec:background-rkhs-kme} to obtain a Kernel Stein
Discrepancy tailored to the multivariate MGGC class. Section~\ref{sec:ksd-choice-operator}
discusses which of the two operators of Section~\ref{sec:stein-ggc} is
retained for the construction of the KSD, Section~\ref{sec:ggc-stein-kernel}
derives the corresponding Stein kernel in closed form, and
Section~\ref{sec:ksd-ggc-convergence} states the two convergence
guarantees -- separation and weak convergence -- that justify the use
of the KSD as a discrepancy measure on $\mathcal{G}_{d}$.

% ---------------------------------------------------------------------
\subsection{Choice of Stein operator}
\label{sec:ksd-choice-operator}
% ---------------------------------------------------------------------

Section~\ref{sec:stein-ggc} provides two Stein operators for
$X \sim \mathcal{G}_{d}(\nu)$: the multivariate Arras
form~\eqref{eq:stein-arras-multi}, with polynomial differential
coefficients, and the MGGC integral form~\eqref{eq:stein-ggc-integral},
with coefficients given by integrals against the Thorin measure $\nu$.

\paragraph{Why not the polynomial form.}
Building a Stein kernel from~\eqref{eq:stein-arras-multi} requires
applying the polynomial differential operator
$\mathcal{T}^{\mathrm{Arras}}_{}$ once on each variable of the
underlying matrix-valued kernel $K(x,y)$. On the discrete sub-class
$\mathcal{G}_{d,n}$, the polynomial $A(t) = \prod_{j=1}^{n}(1 - \im\,\ip{s_{j}}{t})$
has degree $n$; hence the differential operator $A(D)$ is of order
$n$, and the Stein kernel
$\mathcal{T}^{\mathrm{Arras},(x)}_{}\,\mathcal{T}^{\mathrm{Arras},(y)}_{}\, K(x,y)$
involves \emph{all} mixed partial derivatives of $K$ of order up to
$2n$. The number of such partial derivatives in dimension $d$ grows
combinatorially as $\binom{2n + d}{d}$, so that even for moderate
values of $n$ and $d$ -- say $n = 10$ and $d = 5$ -- the evaluation
of the Stein kernel at a single point requires several thousand
distinct partial-derivative calls.

\paragraph{Why the integral form.}
The exponential-jump form~\eqref{eq:stein-ggc-subordinator} avoids the computational cost of evaluating all the derivatives
of $K$. As we shall see in
Section~\ref{sec:ggc-stein-kernel}, applying
$\mathcal{T}^{(x)}_{\nu}\,\mathcal{T}^{(y)}_{\nu}$  to an isotropic kernel
$K(x,y) = k(x,y)\, I_{d}$ produces a closed-form expression involving
only evaluations of $k$ at shifted points and one-dimensional
expectations against $\mathrm{Exp}(1)$. These quantities are uniformly
cheap to compute -- by quadrature or by Monte-Carlo simulation of $E$ -- and
the cost does not depend on the regularity of $\nu$. The integral form
therefore makes the KSD computable at a uniform cost across the entire
class $\mathcal{G}_{d}$, including the absolutely-continuous regime
where the polynomial form is meaningless.

% ---------------------------------------------------------------------
\subsection{Stein kernel for the MGGC class}
\label{sec:ggc-stein-kernel}
% ---------------------------------------------------------------------

We rewrite Definition~\ref{def:stein-kernel} using the operator
$\mathcal{T}_{\nu}$ of equation \eqref{eq:stein-ggc-subordinator} to an isotropic
matrix-valued reproducing kernel $K(x,y) = k(x,y)\, I_{d}$, where $k$
is a scalar reproducing kernel on $\R^{d}$.

\begin{theorem}[Stein kernel for $\mathcal{G}_{d}(\nu)$]
\label{prop:ggc-stein-kernel}
Let $X \sim \mathcal{G}_{d}(\nu)$ with $\nu$ satisfying
condition~\eqref{eq:nu-integrability}, and let $K(x,y) = k(x,y)\, I_{d}$
with $k$ a scalar reproducing kernel on $\R^{d}$. Let $E, E' \sim \mathrm{Exp}(1)$
be independent and independent of $X$. Then the Stein
kernel of Definition~\ref{def:stein-kernel} writes, for every $x, y \in \R^{d}$,
\begin{equation}
\begin{aligned}
  K_{\nu}(x,y) = \mathcal{T}_\nu^{(x)}\mathcal{T}_\nu^{(y)}K(x,y)
  \;&=\; \ip{x}{y}\, k(x,y) \\
       &\quad -\, \int_{\Rp^{d}}\!\ip{x}{s}\,\E\bigl[k(x,\, y + s\, E)\bigr]\,\nu(\mathrm{d}s) \\
       &\quad -\, \int_{\Rp^{d}}\!\ip{s'}{y}\,\E\bigl[k(x + s'\, E',\, y)\bigr]\,\nu(\mathrm{d}s') \\
       &\quad +\, \int_{\Rp^{d}}\!\int_{\Rp^{d}}\!\ip{s'}{s}\,\E\bigl[k(x + s'\, E',\, y + s\, E)\bigr]\,\nu(\mathrm{d}s)\,\nu(\mathrm{d}s').
\end{aligned}
  \label{eq:ggc-stein-kernel}
\end{equation}
For a finite Gamma convolution $X \sim \mathcal{G}_{d,n}(\alpha, s)$
with $\nu = \sum_{j=1}^{n}\alpha_{j}\,\delta_{s_{j}}$,
formula~\eqref{eq:ggc-stein-kernel} simplifies to
\begin{equation}
\begin{aligned}
  K_{\nu}(x,y)
  \;&=\; \ip{x}{y}\, k(x,y)
       - \sum_{j=1}^{n}\alpha_{j}\,\ip{x}{s_{j}}\,\E\bigl[k(x,\, y + s_{j} E)\bigr] \\
       &\quad - \sum_{j=1}^{n}\alpha_{j}\,\ip{s_{j}}{y}\,\E\bigl[k(x + s_{j} E',\, y)\bigr]
       + \sum_{j,j'=1}^{n}\alpha_{j}\,\alpha_{j'}\,\ip{s_{j'}}{s_{j}}\,\E\bigl[k(x + s_{j'} E',\, y + s_{j} E)\bigr].
\end{aligned}
  \label{eq:ggc-stein-kernel-discrete}
\end{equation}
\end{theorem}

\begin{proof}
The proof is a direct computation. We compute first $\mathcal{T}^{(y)}_{\nu} K(x,y)$,  then apply $\mathcal{T}^{(x)}_{\nu}$ to the
resulting vector field.

\smallskip\noindent\textbf{Step 1: action of $\mathcal{T}^{(y)}_{\nu}$.}
Fix $x \in \R^{d}$ and let $j \in \{1,\dots,d\}$. The $j$-th column
of $K(\cdot, \cdot)$, viewed as a vector field in $y$, is the function
$y \mapsto k(x,y)\, e_{j}$ where $e_{j}$ denotes the $j$-th canonical
basis vector. Applying $\mathcal{T}_{\nu}$ in its
form~\eqref{eq:stein-ggc-subordinator} to this vector field gives
\begin{align*}
  \mathcal{T}^{(y)}_{\nu}\!\bigl(k(x,y)\, e_{j}\bigr)
  &\;=\; \ip{y}{e_{j}}\,k(x,y) \,-\, \int_{\Rp^{d}}\!\E\bigl[\ip{s}{e_{j}}\,k(x,\,y + s\,E)\bigr]\,\nu(\mathrm{d}s) \\
  &\;=\; y_{j}\,k(x,y) \,-\, \int_{\Rp^{d}}\! s_{j}\,\E\bigl[k(x,\,y + s\,E)\bigr]\,\nu(\mathrm{d}s).
\end{align*}
Stacking the $d$ columns yields the vector identity
\begin{equation}
  \mathcal{T}^{(y)}_{\nu} K(x,y)
  \;=\; y\,k(x,y) \,-\, \int_{\Rp^{d}}\! s\,\E\bigl[k(x,\,y + s\,E)\bigr]\,\nu(\mathrm{d}s)
  \;\eqqcolon\; \Phi(x,y) \,\in\, \R^{d}.
  \label{eq:proof-Phi}
\end{equation}

\smallskip\noindent\textbf{Step 2: action of $\mathcal{T}^{(x)}_{\nu}$ on $\Phi$.}
Apply $\mathcal{T}_{\nu}$ in the form~\eqref{eq:stein-ggc-subordinator} to the
vector field $x \mapsto \Phi(x, y)$:
\begin{equation}
  K_{\nu}(x,y)
  \;=\; \mathcal{T}^{(x)}_{\nu}\,\Phi(x,y)
  \;=\; \ip{x}{\Phi(x,y)}
       \,-\, \int_{\Rp^{d}}\!\E_{E'}\!\bigl[\ip{s'}{\Phi(x + s'\,E',\, y)}\bigr]\,\nu(\mathrm{d}s'),
  \label{eq:proof-final-stein-kernel}
\end{equation}
where $E' \sim \mathrm{Exp}(1)$ is independent of the variable $E$
appearing in the definition of $\Phi$. From~\eqref{eq:proof-Phi},
\begin{align*}
  \ip{x}{\Phi(x,y)}
  &\;=\; \ip{x}{y}\, k(x,y)
       \,-\, \int_{\Rp^{d}}\!\ip{x}{s}\,\E\bigl[k(x,\,y + s\,E)\bigr]\,\nu(\mathrm{d}s),\\
  \ip{s'}{\Phi(x + s'\,E',\, y)}
  &\;=\; \ip{s'}{y}\,k(x + s'\,E',\, y)
       \,-\, \int_{\Rp^{d}}\!\ip{s'}{s}\,\E_{E}\bigl[k(x + s'\,E',\,y + s\,E)\bigr]\,\nu(\mathrm{d}s).
\end{align*}
Taking the expectation over $E'$ in the second line and substituting
both identities into~\eqref{eq:proof-final-stein-kernel} yields
\begin{align*}
  K_{\nu}(x,y)
  \;=&\; \ip{x}{y}\, k(x,y)
       \,-\, \int_{\Rp^{d}}\!\ip{x}{s}\,\E\bigl[k(x,\,y + s\,E)\bigr]\,\nu(\mathrm{d}s) \\
       &\,-\, \int_{\Rp^{d}}\!\ip{s'}{y}\,\E_{E'}\bigl[k(x + s'\,E',\, y)\bigr]\,\nu(\mathrm{d}s') \\
       &\,+\, \int_{\Rp^{d}}\!\!\int_{\Rp^{d}}\!\ip{s'}{s}\,\E_{E,E'}\bigl[k(x + s'\,E',\, y + s\,E)\bigr]\,\nu(\mathrm{d}s)\,\nu(\mathrm{d}s'),
\end{align*}
which is precisely~\eqref{eq:ggc-stein-kernel}. The
specialisation~\eqref{eq:ggc-stein-kernel-discrete} is obtained by
replacing $\nu$ with $\sum_{j}\alpha_{j}\,\delta_{s_{j}}$.
\end{proof}

\begin{remark}[Practical computation of the expectations of $E$]
\label{rem:expectation-computation}
The one-dimensional expectations in~\eqref{eq:ggc-stein-kernel}, of
the form $\E[k(\cdot,\, \cdot + s\, E)]$ with $E \sim \mathrm{Exp}(1)$,
admit a closed-form expression for several classical kernels (such as
the Laplace and Mat\'ern kernels with low regularity) and can be approximated by Gauss--Laguerre quadrature \cite{DavisRabinowitz1984} or by direct Monte-Carlo
sampling. The choice of the method does not affect the unbiasedness
properties of the resulting empirical KSD estimators.
\end{remark}

% ---------------------------------------------------------------------
\subsection{Theoretical guarantees}
\label{sec:ksd-ggc-convergence}
% ---------------------------------------------------------------------

The KSD construction of Sections~\ref{sec:ksd-choice-operator}--\ref{sec:ggc-stein-kernel}
is purely formal until one establishes that it actually defines a
statistical divergence on a sufficiently large class of target
laws and that the associated estimator inherits the standard
asymptotic guarantees of a $U$-statistic-based $M$-estimator. We now
state three results: a separation property (the KSD is a
statistical divergence, Proposition~\ref{prop:ksd-divergence}), a
strong-consistency / asymptotic-normality theorem for the MSDE
(Theorem~\ref{thm:msde-asymptotics}), and a weak-convergence theorem
for sequences of probability measures
(Theorem~\ref{thm:ksd-weak-convergence}).

\medskip

It is worth emphasising at the outset that these three results have
different scopes. The divergence property
(Proposition~\ref{prop:ksd-divergence}) is purely structural and
applies to any $X$ in $\mathcal{G}_{d}(\nu)$ for
which the Stein operator $\mathcal{T}_{\nu}$ is well-defined. The
strong-consistency statement of
Theorem~\ref{thm:msde-asymptotics} requires a compact parameter
space, and is therefore stated on the constrained set $\Theta_c$ of
Definition~\ref{def:theta-constrained}. The
weak-convergence result of
Theorem~\ref{thm:ksd-weak-convergence}, makes no
use of compactness or identifiability of the parameter space: it
holds on the entire discrete class $\mathcal{G}_{d,n}$. 

\medskip

In what follows, $\theta \in \Theta$ is a parameter of a discrete
model $\mathcal{G}_{d,n}$, $\mathbb{P}_{\theta}$ denotes the
corresponding GGC law, and $\mathcal{T}_{\theta}$ is the Stein
operator~\eqref{eq:stein-ggc-subordinator} with Thorin measure
$\nu_{\theta} = \sum_{j=1}^{n}\alpha_{j}\,\delta_{s_{j}}$. The Stein
kernel of Theorem~\ref{prop:ggc-stein-kernel} is then written
$k_{0,\theta}(x,y)$ to make the parameter dependence explicit. The
associated KSD between an arbitrary probability measure $\mathbb{Q}$
on $\R^{d}$ and the target $\mathbb{P}_{\theta}$ is
\begin{equation}
  \mathrm{KSD}_{\mathcal{T}}(\mathbb{Q}\,\|\,\mathbb{P}_{\theta})
  \;=\;
  \sup_{f \in \mathcal{H}_{K},\,\|f\|_{\mathcal{H}_{K}} \leq 1}
       \bigl|\,\E_{Y \sim \mathbb{Q}}\!\left[(\mathcal{T}_{\theta} f)(Y)\right]\,\bigr|.
  \label{eq:ksd-ggc-def}
\end{equation}

% ---------------------------------------------------------------------
\subsubsection*{Statistical divergence}
% ---------------------------------------------------------------------

The first guarantee is that $\mathrm{KSD}_{\mathcal{T}}$
distinguishes any laws with the target $\mathbb{P}_\theta$ .

\begin{proposition}[KSD as a statistical divergence]
\label{prop:ksd-divergence}
Assume that
\begin{enumerate}[label=(D\arabic*)]
  \item $\mathcal{T}_{\theta}$ is a Stein operator for $\mathbb{P}_{\theta}$
        on the test class $\mathcal{F}$ used in
        Theorem~\ref{lem:stein-ggc};
        \label{ass:D1}
  \item the RKHS $\mathcal{H}_{K} \subseteq \mathcal{F}$, the linear
        functional $f \mapsto \E_{\mathbb{Q}}[(\mathcal{T}_{\theta} f)(X)]$
        is continuous on $\mathcal{H}_{K}$, and $\mathcal{H}_{K}$ is
        dense in $\mathcal{F}$ for the topology that makes this
        functional continuous\footnote{This topology is dictated by the
        Stein operator: because $(\mathcal{T}_{\theta} f)(x)$ contains the
        unbounded linear term $\ip{x}{f(x)}$, the relevant topology is
        that of the weighted supremum norm
        $\|f\|_{\mathcal{F}} \coloneqq \sup_{x \in \R^{d}} \|f(x)\|/(1 + \|x\|)$
        rather than the plain uniform norm. With respect to it, the
        functional $f \mapsto \E_{\mathbb{Q}}[(\mathcal{T}_{\theta} f)(X)]$
        is continuous as soon as $\mathbb{Q}$ has a finite second moment,
        since $|\ip{x}{f(x)}| \leq \|x\|(1+\|x\|)\,\|f\|_{\mathcal{F}}$ is
        then $\mathbb{Q}$-integrable.};
        \label{ass:D2}    
  \item the Stein identity characterises $\mathbb{P}_{\theta}$ on
        $\mathcal{F}$: for every probability measure $\mathbb{Q}$,
        \[
          \E_{\mathbb{Q}}\bigl[(\mathcal{T}_{\theta} f)(X)\bigr] \;=\; 0
          \quad\forall f \in \mathcal{F}
          \quad\Longrightarrow\quad
          \mathbb{Q} \;=\; \mathbb{P}_{\theta}.
        \]
        \label{ass:D3}
\end{enumerate}
Then for every probability measure $\mathbb{Q}$ on $\R^{d}$,
\begin{equation*}
  \mathrm{KSD}_{\mathcal{T}}(\mathbb{Q}\,\|\,\mathbb{P}_{\theta})
  \;=\; 0
  \quad\Longleftrightarrow\quad
  \mathbb{Q} \;=\; \mathbb{P}_{\theta}.
\end{equation*}
\end{proposition}

\begin{remark}[Verification of~\ref{ass:D1} and~\ref{ass:D3} for MGGC laws]
\label{rem:D1-D3-MGGC}
For MGGC distributions, condition~\ref{ass:D1} follows directly
from Theorem~\ref{lem:stein-ggc}. Moreover, the associated Stein identity
determines the characteristic function uniquely, and therefore
characterises the target MGGC law, which proves condition~\ref{ass:D3}.
\end{remark}

\begin{proof}
\smallskip\noindent\textbf{($\Leftarrow$) Trivial direction.}
If $\mathbb{Q} = \mathbb{P}_{\theta}$, then by~\ref{ass:D1} the Stein
identity $\E_{\mathbb{P}_{\theta}}[(\mathcal{T}_{\theta} f)(X)] = 0$ holds
for every $f \in \mathcal{F}$, in particular for every
$f \in \mathcal{H}_{K} \subseteq \mathcal{F}$. The supremum
in~\eqref{eq:ksd-ggc-def} is therefore zero, so
$\mathrm{KSD}_{\mathcal{T}}(\mathbb{P}_{\theta}\,\|\,\mathbb{P}_{\theta}) = 0$.

\smallskip\noindent\textbf{($\Rightarrow$) Non-trivial direction.}
Suppose $\mathrm{KSD}_{\mathcal{T}}(\mathbb{Q}\,\|\,\mathbb{P}_{\theta}) = 0$.
By definition of the supremum on the unit ball of $\mathcal{H}_{K}$,
\begin{equation*}
  \E_{\mathbb{Q}}\bigl[(\mathcal{T}_{\theta} f)(X)\bigr]
  \;=\; 0,
  \qquad \forall f \in \mathcal{H}_{K}.
\end{equation*}
Since $\mathcal{H}_{K}$ is dense in
$\mathcal{F}$ 
by~\ref{ass:D2}, hence:
\begin{equation}
  \E_{\mathbb{Q}}\bigl[(\mathcal{T}_{\theta} f)(X)\bigr]
  \;=\; 0,
  \qquad \forall f \in \mathcal{F}.
  \label{eq:proof-div-F}
\end{equation}

\smallskip\noindent\textbf{Conclusion via~\ref{ass:D3}.}
The Stein identity for $\mathbb{Q}$ on $\mathcal{F}$, established
in~\eqref{eq:proof-div-F}, combined with the characterising
property~\ref{ass:D3}, forces $\mathbb{Q} = \mathbb{P}_{\theta}$.

\end{proof}

% ---------------------------------------------------------------------
\subsubsection*{Strong consistency and asymptotic normality of the MSDE}
% ---------------------------------------------------------------------

\paragraph{Scope.}
The asymptotic theory of the MSDE relies crucially on the
\emph{compactness} of the parameter space and on the
\emph{identifiability} of the parametrisation. As discussed in
Section~\ref{sec:ggc-identifiability}, these two conditions are not satisfied on the unconstrained set
$\R^{n}_{+} \times (\Rp^{d})^{n}$. The constrained parameter space $\Theta_c$, see Definition~\ref{def:theta-constrained}, was introduced to
have identifiability and compactness. The strong-consistency and
asymptotic-normality results below are therefore stated on the
constrained set $\Theta_c$.

\medskip

Given an i.i.d.\ sample $Y_{1}, \dots, Y_{N}$ from a 
distribution $\mathbb{P}_{\theta_{0}}$ with
$\theta_{0} \in \Theta_c$, the empirical $\mathrm{KSD}^{2}$ is a
$U$-statistic
\begin{equation}
  U_{N}(\theta)
  \;\coloneqq\;
  \frac{1}{N(N-1)} \sum_{1 \leq i \neq j \leq N}
       k_{0,\theta}(Y_{i}, Y_{j}),
  \label{eq:U-stat}
\end{equation}
and the MSDE~\eqref{eq:msde} writes
\begin{equation*}
  \widehat{\theta}_{N}
  \;\in\;
  \underset{\theta \in \Theta_c}{\mathrm{argmin}}\;
       U_{N}(\theta).
\end{equation*}

Denote by $M(\theta) \coloneqq \mathrm{KSD}_{\mathcal{T}}^{2}(\mathbb{P}_{\theta_{0}}\,\|\,\mathbb{P}_{\theta})$
the population objective. By Proposition~\ref{prop:ksd-divergence}, if the hypothesis are verified,
together with the identifiability of $\mathcal{G}_{d,n}$ on $\Theta_c$
(Proposition~\ref{prop:gdn-identifiable}), $M(\theta) \geq 0$ with
equality iff $\theta = \theta_{0}$, so that $\theta_{0}$ is the
unique minimiser of $M$ on $\Theta_c$.

\begin{theorem}[Strong consistency and asymptotic normality of the MSDE on $\Theta$ \cite{BarpEtAl2019}]
\label{thm:msde-asymptotics}
Assume Proposition~\ref{prop:ksd-divergence} applies for every
$\theta \in \Theta_c$, and that:
\begin{enumerate}[label=(C\arabic*)]
  \item for every $(x,y) \in \R^{d} \times \R^{d}$, the map
        $\theta \mapsto k_{0,\theta}(x,y)$ is continuous;
        \label{ass:C-continuous}
  \item there exists a measurable envelope
        $H \colon \R^{d} \times \R^{d} \to [0,\infty)$ such that
        $\sup_{\theta \in \Theta_c}|k_{0,\theta}(x,y)| \leq H(x,y)$ for
        all $(x,y)$ and $\E\bigl[H(Y_{1}, Y_{2})\bigr] < \infty$;
        \label{ass:C-envelope}
  \item the model is identifiable on $\Theta_c$
        (Proposition~\ref{prop:gdn-identifiable}).
        \label{ass:C-id}
\end{enumerate}
Then the strong consistency
\begin{equation*}
  \widehat{\theta}_{N}
  \;\xrightarrow[N \to \infty]{\textnormal{a.s.}}\;
  \theta_{0}
\end{equation*}
holds. Suppose furthermore that, on a compact neighbourhood
$\mathcal{N} \subseteq \Theta_c$ of $\theta_{0}$, the map
$\theta \mapsto k_{0,\theta}(x,y)$ is of class $C^{2}$ and satisfies
\begin{enumerate}[label=(N\arabic*)]
  \item
        $\E\bigl[
        \|\nabla_{\theta} k_{0,\theta_{0}}(Y_{1},Y_{2})\|^{2}
        \bigr] < \infty$;
        \label{ass:N-grad-L2}

  \item there exist envelopes\footnote{An \emph{envelope} (or
        \emph{envelope function}) for a family
        $\{g_{\theta}\}_{\theta \in \Theta}$ of measurable functions is
        a single measurable function $H$ that dominates the whole
        family pointwise,
        $\sup_{\theta \in \Theta}|g_{\theta}(\cdot)| \leq H(\cdot)$.
        The integrability of $H$ then transfers uniformly in $\theta$
        to every member of the family, which is what allows the
        dominated-convergence and uniform-law-of-large-numbers
        arguments to go through.}
        $H_{1},H_{2}$ such that
        \(
          \sup_{\theta \in \mathcal{N}}
          \|\nabla_{\theta}k_{0,\theta}(x,y)\|
          \leq H_{1}(x,y),
\          \E\bigl[H_{1}(Y_{1},Y_{2})^{2}\bigr] < \infty,
        \)
        and
        \(
          \sup_{\theta \in \mathcal{N}}
          \|\Hess_{\theta}k_{0,\theta}(x,y)\|
          \leq H_{2}(x,y),
          \
          \E\bigl[H_{2}(Y_{1},Y_{2})\bigr] < \infty;
        \)
        \label{ass:N-envelopes}

  \item the population objective, defined as the limit
        \(
          M(\theta) \;=\; \lim_{N \to \infty}U_{N}(\theta),
        \)
        is twice differentiable at $\theta_{0}$ and its Hessian,
        $\Hess_{\theta}M(\theta_{0})$, is non-singular
        (equivalently, positive definite).
        \label{ass:N-hessian}
\end{enumerate}
Then the asymptotic normality
\begin{equation}
  \sqrt{N}\,\bigl(\widehat{\theta}_{N}-\theta_{0}\bigr)
  \;\xrightarrow[N \to \infty]{d}\;
  \mathcal{N}\!\left(
    0,\,
    \Hess_{\theta}M(\theta_{0})^{-1}
    (4\Sigma)
    \Hess_{\theta}M(\theta_{0})^{-1}
  \right)
  \label{eq:asymptotic-normality}
\end{equation}
holds, where
\(
  \Sigma = \mathrm{Var}\bigl(\psi(Y_{1})\bigr)
\)
and
\(
  \psi(x)
  \coloneqq
  \E\bigl[
    \nabla_{\theta}k_{0,\theta_{0}}(x,Y_{2})
  \bigr].
\)
\end{theorem}

\begin{remark}
For the model $\mathcal{G}_{d,n}$ on $\Theta_c$, assumptions \emph{(C1)-(C2)} and \emph{(N1)-(N2)-(N3)} hold for the Gaussian kernel. The compactness of $\Theta_c$ yields uniform bounds on the parameters, while the boundedness and smoothness of the Gaussian kernel ensure integrability, continuity, and envelopes for the Stein kernel, his gradient, and his Hessian. \emph{(C3)} is automatically satisfied by Proposition \ref{prop:gd-identifiable}.
\end{remark}

\begin{proof}
The argument proceeds in two stages, following the standard
$M$-estimator pattern based on uniform laws of large numbers and the
central limit theorem for $U$-statistics \cite{Hoeffding1948}.
For more details, see \cite[Theorem~4]{BarpEtAl2019}.
\end{proof}

The two conditions~\ref{ass:N-grad-L2} and~\ref{ass:N-hessian} play
complementary roles in the asymptotic-normality statement. Condition~\ref{ass:N-grad-L2}, the square-integrability of the
parameter-gradient of the Stein kernel, is what makes the rescaled
gradient\(\sqrt{N}\,\nabla_{\theta}U_{N}(\theta_{0}),\)
which is a second-order $U$-statistic with kernel
$\nabla_{\theta}k_{0,\theta_{0}}$, amenable to Hoeffding's central
limit theorem. A finite second moment yields a Gaussian limit with
asymptotic covariance $4\Sigma$. Condition~\ref{ass:N-hessian}, the non-singularity of the population
Hessian
\(
  \Hess_{\theta}M(\theta_{0}),
\)
is a local curvature and local-identifiability requirement at the true
parameter. It ensures that $\theta_{0}$ is a non-degenerate minimum of
the population objective $M$, so that the Hessian can be inverted in
the Taylor expansion of the first-order optimality condition. Together, these two conditions turn the linearisation
\(
  \sqrt{N}\bigl(\widehat{\theta}_{N}-\theta_{0}\bigr)
  =
  -\bigl[\Hess_{\theta}U_{N}\bigr]^{-1}
  \sqrt{N}\,\nabla_{\theta}U_{N}(\theta_{0})
\)
into the sandwich limit~\eqref{eq:asymptotic-normality} via Slutsky's
lemma: Condition~\ref{ass:N-grad-L2} controls the random gradient
factor, while Condition~\ref{ass:N-hessian} controls the deterministic
matrix factor that rescales it.
% ---------------------------------------------------------------------
\subsubsection*{Weak convergence from the KSD}
% ---------------------------------------------------------------------

\paragraph{Scope.}
In contrast to Theorem~\ref{thm:msde-asymptotics}, the
weak-convergence theorem below makes no use of compactness or
identifiability of the parameter space, since it concerns sequences
of \emph{probability measures} on $\R^{d}$ rather than parameters estimation on
$\Theta_c$. It applies to any target $\mathbb{P}_{}$ in the unconstrained class $\mathcal{G}_{d}$, and serves
as the topological complement to the divergence property of
Proposition~\ref{prop:ksd-divergence}: the latter says that
$\mathrm{KSD}_{\mathcal{T}}$ \emph{separates} laws, while the former
says that it \emph{controls weak convergence} as soon as the
candidate sequence is tight and uniformly integrable. The result
adapts to the GGC setting the convergence-determining theorems of
\citet{GorhamMackey2017} for Langevin Stein operators and of
\citet{BarpEtAl2019} for matrix-valued and diffusion-based Stein
operators.

\begin{theorem}[Weak convergence from the KSD on $\mathcal{G}_{d}$]
\label{thm:ksd-weak-convergence}
Let $\nu$ be a Thorin measure on $\Rp^{d}$ and let
$\mathbb{P}_{\nu} \in \mathcal{G}_{d}(\nu)$ be the target GGC
distribution, with integral Stein operator $\mathcal{T}_{\nu}$ of
Theorem~\ref{lem:stein-ggc}. Suppose the matrix-valued kernel
$K(x,y) = k(x,y)\, I_{d}$ is built such that
$\mathrm{KSD}_{\mathcal{T}}(\,\cdot\,\|\,\mathbb{P}_{\nu})$ is a
statistical divergence (Proposition~\ref{prop:ksd-divergence}). Let
$(\mathbb{Q}_{N})_{N \geq 1}$, with
$\mathbb{Q}_{N} = \mathbb{P}_{\nu_{N}} \in \mathcal{G}_{d}(\nu_{N})$, be a
sequence of GGC laws whose Thorin measures $\nu_N$satisfy the uniform
second-moment bound
\begin{equation}
  \sup_{N \geq 1} \int_{\Rp^{d}} \bigl(\|s\| + \|s\|^{2}\bigr)\,\nu_{N}(\mathrm{d}s)
  \;<\; \infty.
  \label{eq:thorin-moment-bound}
\end{equation}
Then
\begin{equation}
  \mathrm{KSD}_{\mathcal{T}}(\mathbb{Q}_{N}\,\|\,\mathbb{P}_{\nu})
  \;\xrightarrow[N \to \infty]{}\; 0
  \quad\Longrightarrow\quad
  \mathbb{Q}_{N}
  \;\xRightarrow[N \to \infty]{}\;
  \mathbb{P}_{\nu}.
  \label{eq:ksd-weak}
\end{equation}
\end{theorem}

\begin{remark}
In the discrete case $\mathcal{G}_{d,n}$, under the assumptions of the theorem, the implication \eqref{eq:ksd-weak} holds on the entire unconstrained parameter space. Classics kernels such as Gaussian kernel, IMQ kernel or Matérn kernel satisfies all the hypothesis.
\end{remark}

\begin{proof}
The proof adapts to the GGC operator the convergence-determining
argument of \citet{GorhamMackey2017}. Departure point is that tightness and uniform integrability are not
assumed: they are \emph{derived} in Step~1 from the structural
bound~\eqref{eq:thorin-moment-bound}, using the fact that every
candidate $\mathbb{Q}_{N}$ is itself a GGC law.

\smallskip\noindent\textbf{Step 1: tightness and uniform moments from
the Thorin bound.}
Differentiating the log-characteristic function~\eqref{eq:cf-Gd} of
$\mathbb{Q}_{N} = \mathbb{P}_{\nu_{N}}$ gives
\[
  \nabla_{t}\log\charf_{\nu_{N}}(t)
  \;=\; \im\!\int_{\Rp^{d}}\!\frac{s}{1 - \im\,\ip{s}{t}}\,\nu_{N}(\mathrm{d}s),
  \qquad
  \nabla_{t}^{2}\log\charf_{\nu_{N}}(t)
  \;=\; -\!\int_{\Rp^{d}}\!\frac{s\,s^{\top}}{(1 - \im\,\ip{s}{t})^{2}}\,\nu_{N}(\mathrm{d}s).
\]
Evaluating at $t = 0$, the mean vector and the covariance matrix of
$\mathbb{Q}_{N}$ are
\[
  m_{\mathbb{Q}_{N}} \;\coloneqq\; \E_{\mathbb{Q}_{N}}[X] \;=\; \int_{\Rp^{d}}\! s\,\nu_{N}(\mathrm{d}s),
\]
$$
 \Sigma_{\mathbb{Q}_{N}}
  \;\coloneqq\; \E_{\mathbb{Q}_{N}}\!\bigl[(X - m_{\mathbb{Q}_{N}})(X - m_{\mathbb{Q}_{N}})^{\top}\bigr]
  \;=\; \int_{\Rp^{d}}\! s\,s^{\top}\,\nu_{N}(\mathrm{d}s) \;\in\; \R^{d \times d}.
$$
Since $\E_{\mathbb{Q}_{N}}\|X\|^{2} = \|m_{\mathbb{Q}_{N}}\|^{2} + \mathrm{tr}\,\Sigma_{\mathbb{Q}_{N}}$,
\[
  \E_{\mathbb{Q}_{N}}\!\bigl[\|X\|^{2}\bigr]
  \;\leq\; \Bigl(\int_{\Rp^{d}}\!\|s\|\,\nu_{N}(\mathrm{d}s)\Bigr)^{2}
       + \int_{\Rp^{d}}\!\|s\|^{2}\,\nu_{N}(\mathrm{d}s),
\]
which because of ~\eqref{eq:thorin-moment-bound} is bounded by a constant
$C < \infty$ independent of $N$:
\begin{equation}
  \sup_{N \geq 1}\E_{\mathbb{Q}_{N}}\!\bigl[\|X\|^{2}\bigr] \;\leq\; C \;<\; \infty.
  \label{eq:uniform-second-moment}
\end{equation}
Markov's
inequality gives $\sup_{N}\mathbb{Q}_{N}(\|X\| > R) \leq C/R^{2} \to 0$
as $R \to \infty$, so $(\mathbb{Q}_{N})_{N \geq 1}$ is tight.

\smallskip\noindent\textbf{Step 2: extraction of a convergent
subsequence.}
By Prokhorov's theorem and Step~1, every subsequence of $(\mathbb{Q}_{N})$
admits a further subsequence -- still denoted $(\mathbb{Q}_{N_{j}})$ --
converging weakly to some probability measure $\mathbb{Q}_{\infty}$ on
$\R^{d}$.

\smallskip\noindent\textbf{Step 3: passing the Stein identity to the
limit.}
Fix $f \in \mathcal{H}_{K}$ with $\|f\|_{\mathcal{H}_{K}} \leq 1$. The
hypothesis together with the variational definition~\eqref{eq:ksd-ggc-def}
gives $\E_{\mathbb{Q}_{N}}[(\mathcal{T}_{\nu} f)(Y)] \to 0$. Recall that
\[
  (\mathcal{T}_{\nu} f)(x)
  \;=\; \ip{x}{f(x)} \,-\, \int_{\Rp^{d}}\!\E_{E}\bigl[\ip{s}{f(x + s\, E)}\bigr]\,\nu(\mathrm{d}s).
\]
The integral term is a bounded continuous function of $x$ (bounded by
$\displaystyle \|f\|_{\infty}\int\|s\|\,\nu(\mathrm{d}s) < \infty$, as $f$ embeds in the
bounded continuous functions since $k$ is bounded), and so passes to the
limit by the Portmanteau theorem. The linear term $\ip{x}{f(x)}$ is
continuous but unbounded; from $|\ip{x}{f(x)}| \leq \|f\|_{\infty}\|x\|$
and the uniform second moment~\eqref{eq:uniform-second-moment}, the
family $\{x \mapsto \ip{x}{f(x)}\}$ is uniformly integrable along
$(\mathbb{Q}_{N})$. The continuous-mapping theorem combined with uniform
integrability then upgrades $\mathbb{Q}_{N_{j}} \Rightarrow \mathbb{Q}_{\infty}$
to convergence of the corresponding integrals, so that
\[
  \E_{\mathbb{Q}_{N_{j}}}\!\bigl[(\mathcal{T}_{\nu} f)(Y)\bigr]
  \;\xrightarrow[j \to \infty]{}\;
  \E_{\mathbb{Q}_{\infty}}\!\bigl[(\mathcal{T}_{\nu} f)(Y)\bigr].
\]

\smallskip\noindent\textbf{Step 4: identification of the limit.}
Combining the two previous displays,
$\E_{\mathbb{Q}_{\infty}}[(\mathcal{T}_{\nu} f)(Y)] = 0$ for every
$f \in \mathcal{H}_{K}$ with $\|f\|_{\mathcal{H}_{K}} \leq 1$, hence (by
homogeneity) for every $f \in \mathcal{H}_{K}$. Since $\mathrm{KSD}_{\mathcal{T}}$ is a statistical
divergence (Proposition~\ref{prop:ksd-divergence}), it forces
$\mathbb{Q}_{\infty} = \mathbb{P}_{\nu}$.

\smallskip\noindent\textbf{Step 5: convergence of the whole sequence.}
Every subsequential weak limit of $(\mathbb{Q}_{N})$ equals
$\mathbb{P}_{\nu}$; combined with tightness (Step~1), the standard
sub-subsequence argument yields $\mathbb{Q}_{N} \Rightarrow \mathbb{P}_{\nu}$
as $N \to \infty$.
\end{proof}

% =====================================================================
\section{Numerical applications}
\label{sec:numerical-applications}
% =====================================================================

The aim of this section is to test the methodology developed in the
previous sections on synthetic data. We start with a study of the
identifiability of the parameters (Section~\ref{sec:num-identifiability}),
which emphazise the role of the constrained parameter space $\Theta$
introduced in Definition~\ref{def:theta-constrained} and motivates
the practical choices made in the subsequent experiments. Throughout this section, we work with the isotropic matrix-valued
reproducing kernel
\begin{equation}
  K(x,y) \;=\; I_{d}\, k(x,y),
  \qquad
  k(x,y) \;=\; \exp\!\left(-\frac{\|x - y\|^{2}}{2\sigma^{2}}\right),
  \label{eq:rbf-kernel}
\end{equation}
where $k$ is the Gaussian (RBF) kernel of bandwidth $\sigma > 0$. This
choice satisfies all the regularity assumptions invoked in
Sections~\ref{sec:ggc-stein-kernel} and~\ref{sec:ksd-ggc-convergence}:
$k$ is bounded, continuous, translation-invariant, $C^{\infty}$,
characteristic, and $C_{0}$-universal on $\R^{d}$
\cite{Sriperumbudur2010}, so that
Theorems~\ref{thm:msde-asymptotics} and~\ref{thm:ksd-weak-convergence}
both apply. Other bounded characteristic kernels, such as sufficiently smooth Matérn or inverse multiquadric (IMQ) kernels, could also be used.

\paragraph{Algorithmic complexity.}
Evaluating the empirical objective $U_{N}(\theta)$ in~\eqref{eq:U-stat}
pass by evaluating the Stein kernel $k_{0,\theta}(Y_{i}, Y_{j})$ at
every ordered pair $(i,j)$ with $i \neq j$, that is, $N(N-1)$ kernel
evaluations. With the discrete-case
formula~\eqref{eq:ggc-stein-kernel-discrete}, each evaluation is a
sum of $(n+1)^{2} = 1 + 2n + n^{2}$ terms: a single diagonal term
$\ip{x}{y}\,k(x,y)$, $2n$ single-jump terms (two per atom of the Thorin
measure), and $n^{2}$ double-jump terms (one per ordered pair of atoms),
each term involving a one-dimensional expectation against an
$\mathrm{Exp}(1)$ variable. Computing $U_{N}(\theta)$ at a
single $\theta$ therefore costs
$\mathcal{O}\!\bigl(\,N^{2}\, n^{2}\, q\,\bigr)$
operations, where $q$ is the number of quadrature nodes (or
Monte-Carlo samples) used to approximate each
expectation\footnote{For the isotropic Gaussian
kernel~\eqref{eq:rbf-kernel}, the single-jump expectations
$\E\bigl[k(x,\, y + s_{j} E)\bigr]$ and
$\E\bigl[k(x + s_{j} E,\, y)\bigr]$ appearing
in~\eqref{eq:ggc-stein-kernel-discrete} are available in closed form,
because $k(x, y + s\, E)$ is a univariate Gaussian random variable in the scalar
random variable $s\, E$ -- which integrates explicitly against the
exponential density (see appendix \ref{app:gaussian-single-jump}). The double-jump expectation
$\E\bigl[k(x + s_{j'} E',\, y + s_{j} E)\bigr]$ is, however, no longer
available in closed form, since the integrand becomes a function of
the two-dimensional variable $(s_{j'} E', s_{j} E)$ whose joint
distribution does not have a Gaussian-conjugate structure with the
RBF kernel. We therefore compute the single-jump expectations
analytically (for speed and accuracy) and the double-jump expectations
by Gauss--Laguerre quadrature with $q$ nodes per integral, i.e.\ $q^{2}$
nodes for each double expectation.}.
% ---------------------------------------------------------------------
\subsection{Parameter identifiability}
\label{sec:num-identifiability}
% ---------------------------------------------------------------------

The first experiment compares the behaviour of the MSDE on the
constrained parameter space $\Theta_c$ of
Definition~\ref{def:theta-constrained} against its behaviour on the
unconstrained set $\R^{n}_{*} \times (\Rp^{d})^{n}$. The aim is to
illustrate empirically the theoretical fact that strong consistency
(Theorem~\ref{thm:msde-asymptotics}) holds on $\Theta_c$ but \emph{not}
on the unconstrained set.

\paragraph{Experimental setup.}
We work in dimension $d = 1$ with $n = 2$ atoms. The data-generating
process is $Y_{1}, \dots, Y_{N} \overset{\text{i.i.d.}}{\sim} \mathbb{P}_{\theta_{0}}$,
where $\theta_{0} = (\alpha_{0}, s_{0})$ is the
$\mathcal{G}_{1,2}$-parameter
\[
  \alpha_{0} \;=\; (4,\, 4),
  \qquad
  s_{0} \;=\; (2,\, 6).
\]
The sample size is fixed at $N = 3000$. The MSDE
$\widehat{\theta}_{N}$ is computed twice on each sample: once on the
compact constrained space $\Theta_c$ -- with separation parameter
$\delta$ chosen below the actual gap between the two true atoms
-- and once on the unconstrained set
$\R^{2}_{*} \times \R^{2}_{*}$. To assess the variability of the
estimator, the entire procedure is repeated independently on
$60$ different samples, yielding $60$ values of $\widehat{\theta}_{N}$
in each setting. The two resulting empirical distributions are
compared, parameter by parameter.

\paragraph{Numerical results.}
Figure~\ref{fig:identif-boxplots} reports the boxplots; the contrast is
striking. On $\Theta_c$ (blue), the four estimators concentrate sharply
around the true values $\alpha_{0}$ and $s_{0}$ -- the boxes are barely
visible, an almost-deterministic behaviour consistent with the
strong-consistency statement of Theorem~\ref{thm:msde-asymptotics}. On the
unconstrained set (orange), all four exhibit broad distributions: biased
medians, inter-quartile ranges spanning several units, and whiskers several
times larger than the true parameters. Strong consistency of the MSDE
clearly fails outside $\Theta_c$.

\begin{figure}[t]
  \centering
  \includegraphics[width=0.95\linewidth]{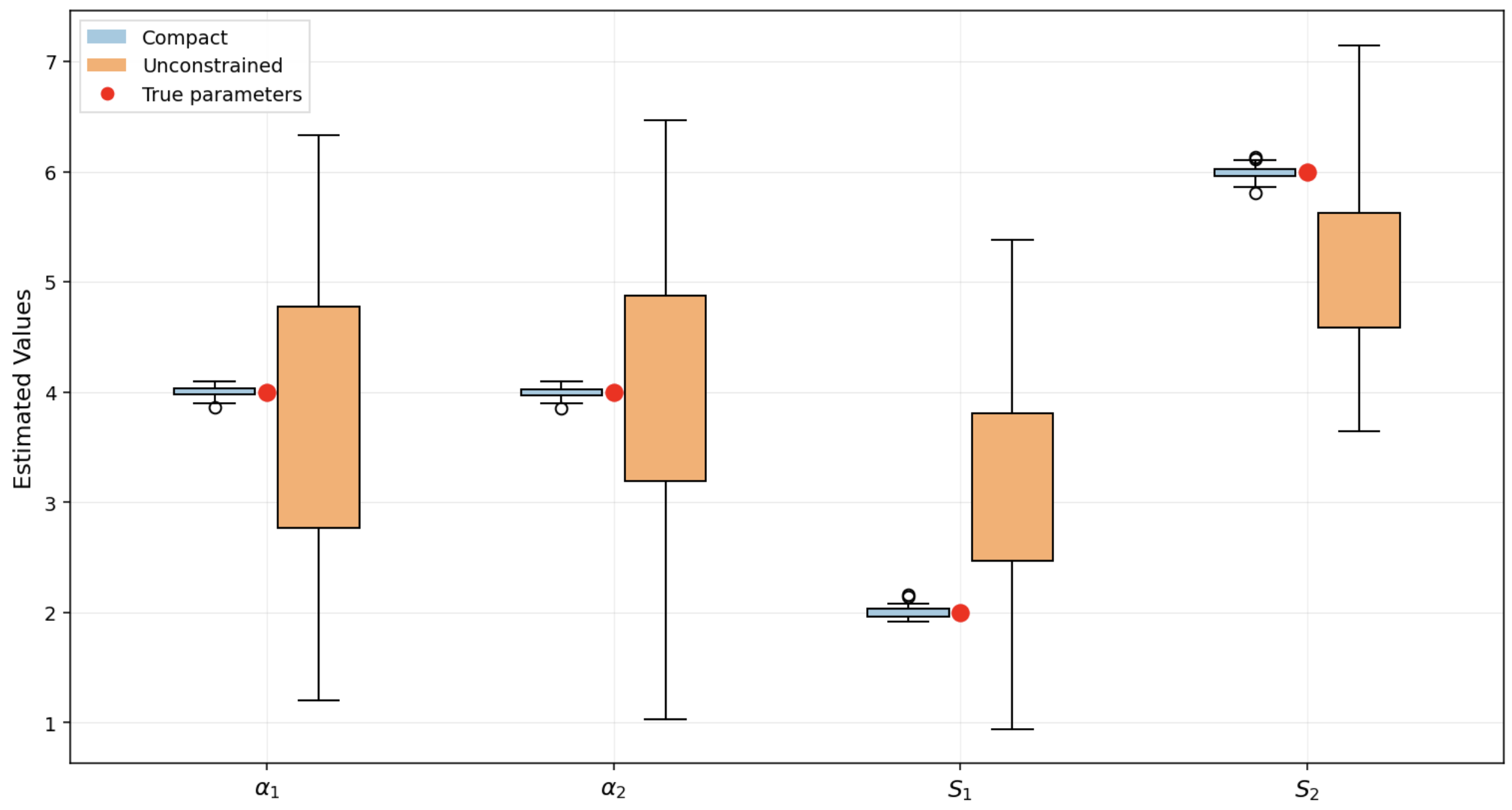}
  \caption{Boxplots of the MSDE $\widehat{\theta}_{N}$ across $60$
  independent samples of size $N = 3000$ drawn from a
  $\mathcal{G}_{1,2}$ distribution with true parameters
  $\alpha_{0} = (4, 4)$ and $s_{0} = (2, 6)$ (red dots). Blue (left):
  estimator computed on the compact constrained set $\Theta_c$ of
  Definition~\ref{def:theta-constrained}. Orange (right): estimator computed
  on the unconstrained set $\R^{2}_{*} \times \R^{2}_{*}$.}
  \label{fig:identif-boxplots}
\end{figure}
% ---------------------------------------------------------------------
\subsection{Comparison with the Laguerre projection method}
\label{sec:num-laverny-comparison}
% ---------------------------------------------------------------------

The second experiment compares the MSDE on the GGC class with the
Laguerre projection method of \citet{LavernyEtAl2021}, which
estimates a target density on $\Rp^{d}$ by projecting it onto a
truncated tensorised Laguerre basis and recovering the GGC parameters
from the sample moments via a moment-matching procedure. For the
KSD-based method, two operators of Section~\ref{sec:stein-ggc} are
tested: the polynomial-coefficient \emph{differential operator
kernel}~\eqref{eq:stein-arras-discrete} and the
\emph{exponential-jump kernel}~\eqref{eq:stein-ggc-discrete}. The
explicit form of the Stein kernel obtained by applying the
polynomial operator to the Gaussian RBF $K(x,y) = I_{d}\, k(x,y)$ is
derived in Appendix; the formula is
concise but involves multivariate Hermite polynomials and is
substantially more expensive to evaluate than its exponential-jump
counterpart of Proposition~\ref{prop:ggc-stein-kernel}.  Estimations are made on $\Theta$.

\paragraph{Experimental setup.}
The target distribution is the lognormal $\mathrm{LN}(0, 0.83)$,
which belongs to $\mathcal{G}_{1}$ but lies outside any finite
sub-class $\mathcal{G}_{1,n}$ (Example~\ref{ex:pareto-lognormal}).
Each method approximates it by a five-atom GGC, i.e.\ by
$\mathbb{P}_{\widehat{\theta}_{N}} \in \mathcal{G}_{1,5}$. The
sample size is $N = 3000$.

\begin{figure}[h]
  \centering
  \begin{tabular}{cc}
    \includegraphics[width=0.45\linewidth]{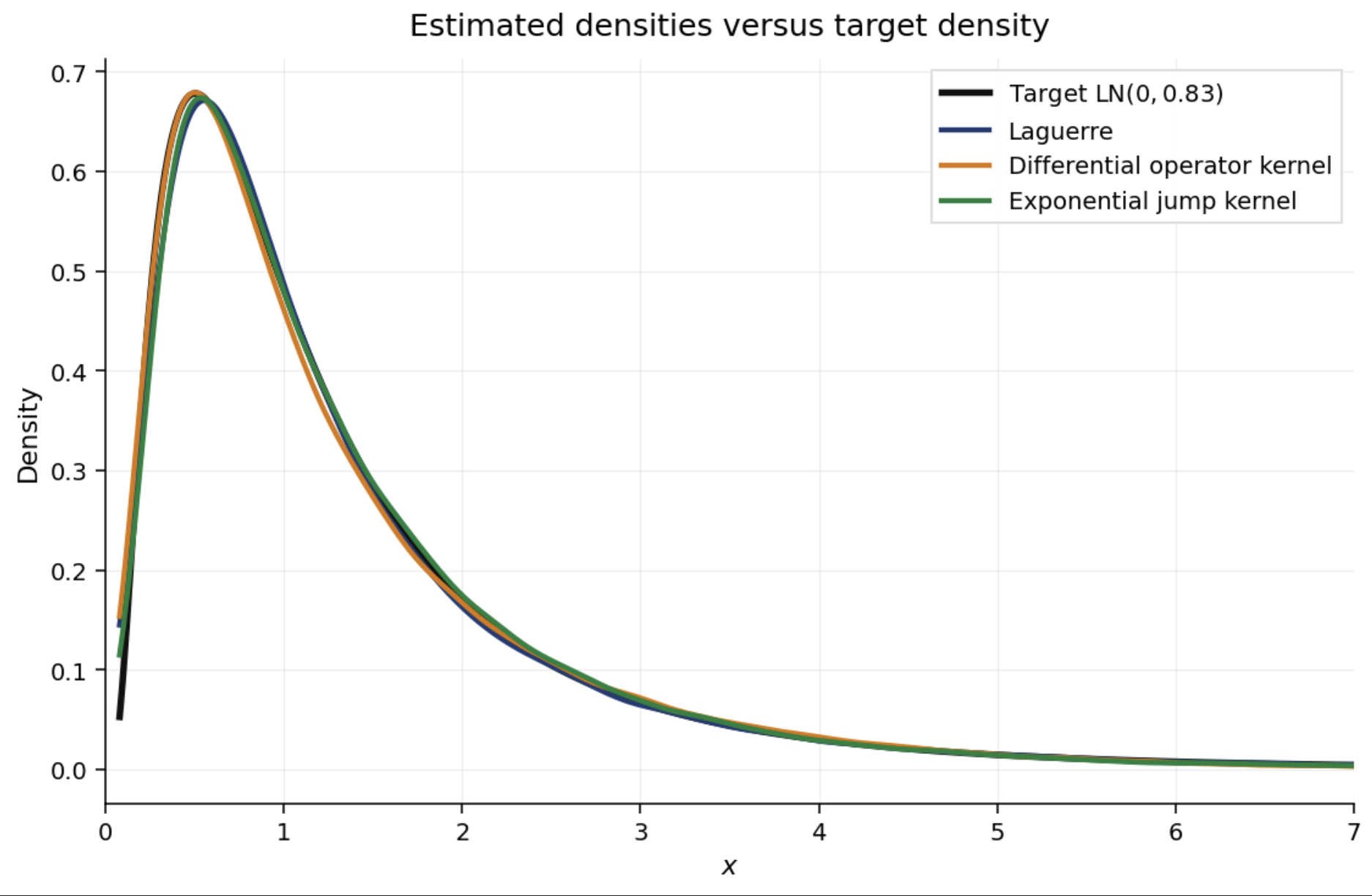} &
    \includegraphics[width=0.45\linewidth]{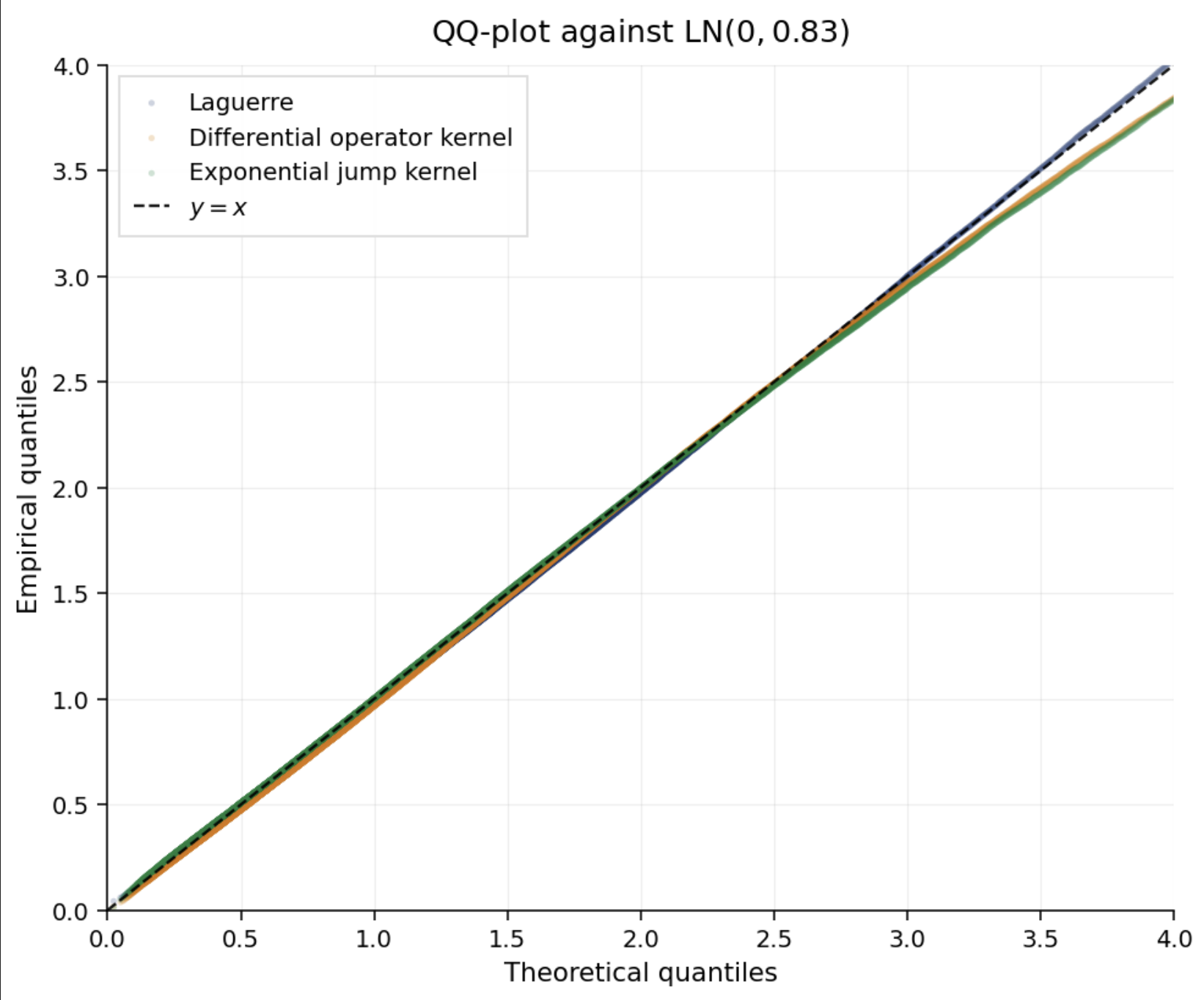} \\
    \includegraphics[width=0.45\linewidth]{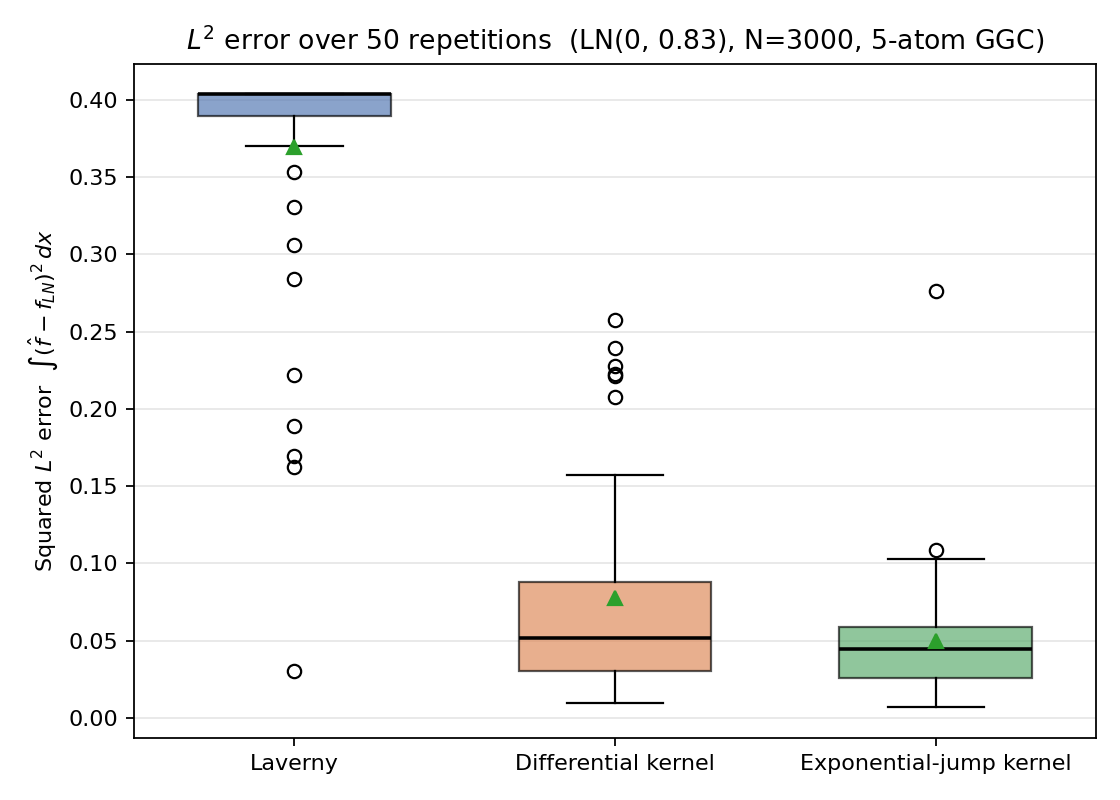} &
    \includegraphics[width=0.45\linewidth]{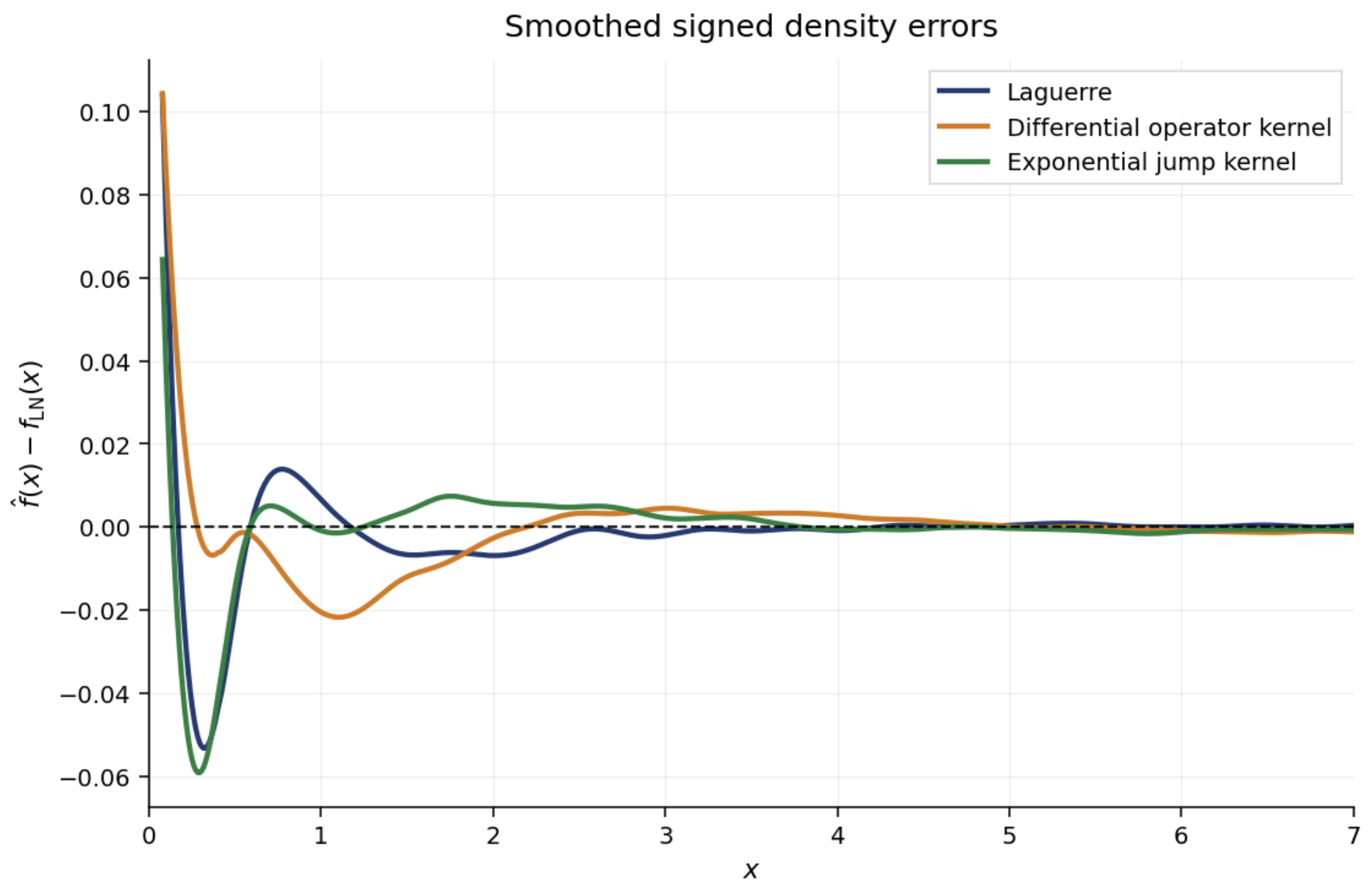} \\
  \end{tabular}
  \caption{Comparison of the Laguerre projection method
  \cite{LavernyEtAl2021} with the MSDE based on the
  \emph{differential operator kernel} and on the
  \emph{exponential-jump kernel}, for a five-atom GGC fit to the
  $\mathrm{LN}(0, 0.83)$ target with sample size $N = 3000$. Top-left:
  estimated densities versus target. Top-right: QQ-plots against
  the target. Bottom-right: smoothed signed density errors
  $\widehat{f}(x) - f_{\mathrm{LN}}(x)$. Bottom-left: boxplot of the squared
  $L^{2}$ density errors.}
  \label{fig:laverny-comparison}
\end{figure}

\paragraph{Numerical results.}
The four panels of Figure~\ref{fig:laverny-comparison} report the
estimated densities versus the target (top-left), the QQ-plots
against the target (top-right), the smoothed signed density errors
$\widehat{f}(x) - f_{\mathrm{LN}}(x)$ (bottom-left), and the squared
$L^{2}$ density errors (bottom-right) for each of the three methods.
The three densities are visually very close to the target on
$[0, 7]$, and the QQ-plots align with the diagonal $y = x$ on
$[0, 4]$. The squared $L^{2}$ error boxplots show that the
\emph{exponential-jump kernel} achieves the smallest average error, closely followed by the \emph{differential operator kernel}, with the
Laguerre method slightly behind. The
KSD-based methods therefore outperform Laguerre on the bulk of the
distribution. The smoothed signed errors (bottom-left) and the
right-end of the QQ-plot, however, show that the Laguerre estimator
tracks the target slightly better in the upper tail
($x \gtrsim 3$): the lognormal tail is captured more accurately by
the Laguerre basis than by either Stein kernel, which slightly
underestimates the upper tail. A possible explanation is that the KSD is a expectation-based criterion, so discrepancies in tail receive less weight than those in the bulk.

\subsection{Three-dimensional application}

We now illustrate the proposed estimation procedure on a genuinely multivariate example. The
dataset is simulated in dimension \(d=3\). Each marginal distribution is Lognormal$(0,1)$, and the
dependence structure is introduced through a Gumbel copula with parameter \(\rho=1.7\). We fit a finite multivariate gamma convolution model of
\(
\mathcal{G}_{3,30}
\).  The estimation is made on $\Theta$.
\begin{figure}[h]
    \centering
    \includegraphics[width=\textwidth]{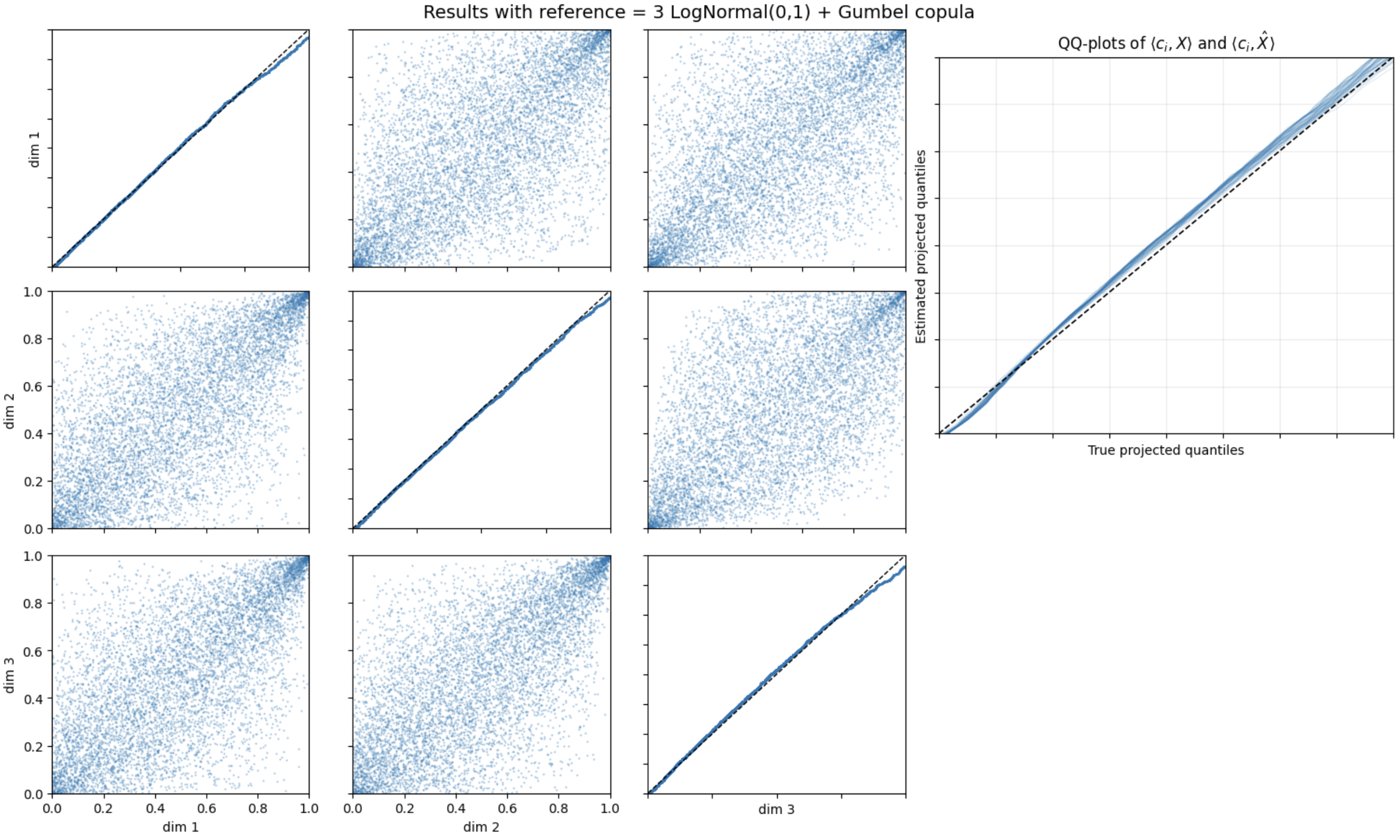}
    \caption{Comparison between the simulated three-dimensional lognormal sample with Gumbel copula dependence and the fitted \(G_{3,30}\) model. Left: pairwise scatter plots and marginal QQ-plots. Right: projected QQ-plots comparing \(\langle c_i,X\rangle\) and \(\langle c_i,\widehat X\rangle\) along 50  random projections directions $c_i$.}
    \label{fig:application_3d}
\end{figure}

Figure~\ref{fig:application_3d} compares the simulated sample \(X\) with the fitted model
\(\widehat X\). The matrix of pairwise scatter plots shows that the fitted distribution reproduces
both the marginal shapes and the dependence structure of the data. In particular, the increasing
dependence induced by the Gumbel copula is clearly recovered across the three pairs of coordinates.
The diagonal QQ-plots indicate that the marginal distributions are well approximated, with only
minor deviations in the upper tail.

The projected QQ-plot provides a complementary multivariate diagnostic by comparing the
quantiles of \(\langle c_i,X\rangle\) and \(\langle c_i,\widehat X\rangle\) for 50 projection directions
\(c_i\). The curves remain close to the diagonal, which confirms that the fitted \(G_{3,30}\) model
captures the distribution not only marginally, but also through its multivariate projections.

\subsection{Execution time and scalability}

We finally investigate the computational behaviour of the proposed estimation procedure when
the dimension and the number of atoms increase. The benchmark is performed for several values
of the dimension \(d\) and of the number of atoms \(n_{\mathrm{atoms}}\). For a model in
\(G_{d,n_{\mathrm{atoms}}}\), the total number of parameters is \(
n_{\mathrm{atoms}}(d+1).
\)

\begin{figure}[h]
    \centering
    \includegraphics[width=1\textwidth]{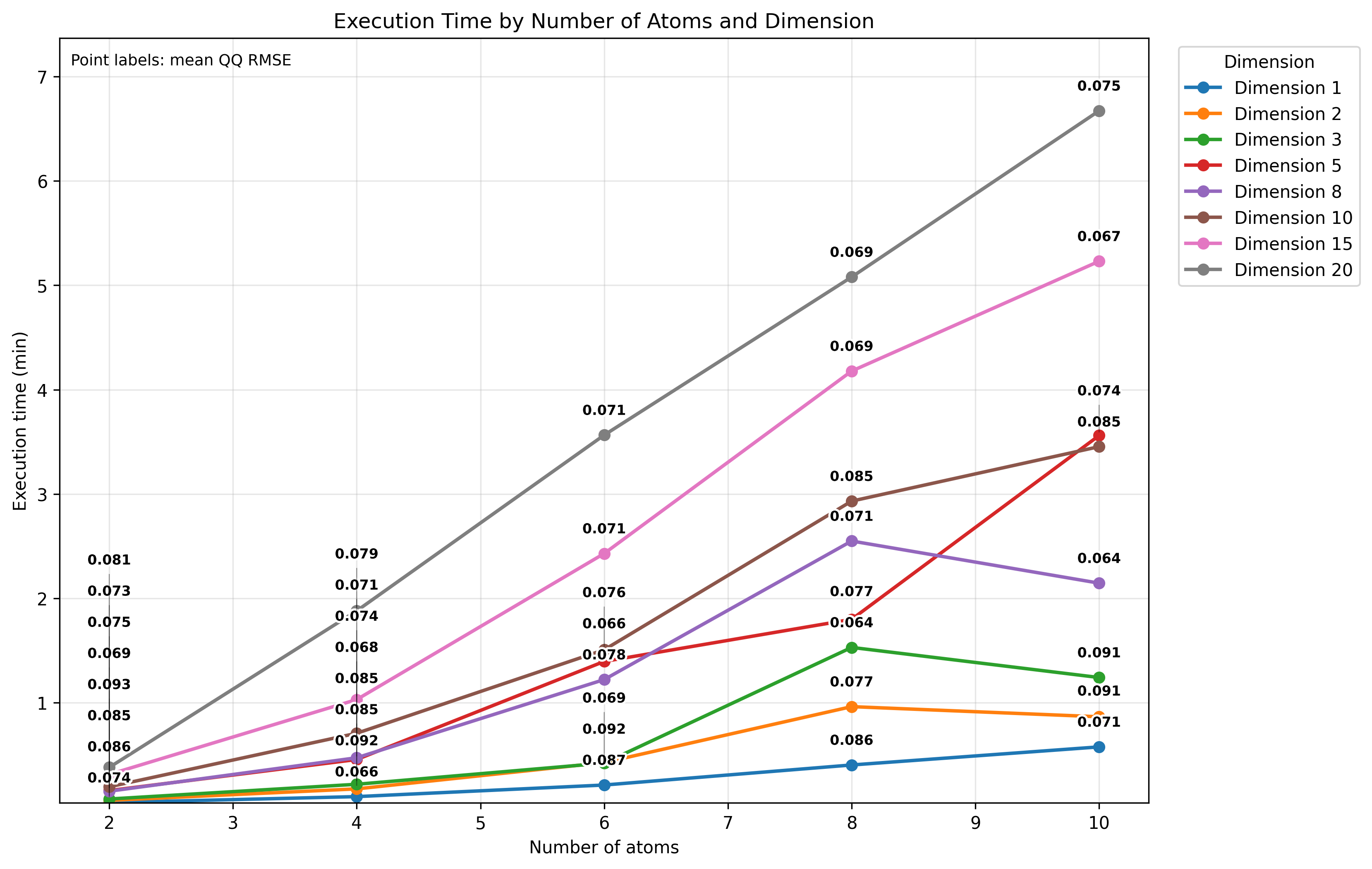}
    \caption{Average execution time as a function of the dimension \(d\), for several numbers of atoms for $N=3000$. For each point, the associated QQ-RMSE on 10 i.i.d samples.}
    \label{fig:time_dimension}
\end{figure}
\paragraph{Estimation protocol.} Each point of Figure~\ref{fig:time_dimension} is
obtained as follows. For a fixed dimension $d$ and a fixed number of atoms
$n_{\mathrm{atoms}}$, we draw an i.i.d.\ sample from a target law
$\mathcal{G}_{d,n_{\mathrm{atoms}}}(\alpha_0,s_0)$ and we fit it with a model taken
from the \emph{same} class $\mathcal{G}_{d,n_{\mathrm{atoms}}}$, i.e.\ a finite
gamma convolution with the same dimension and the same number of atoms as the
data-generating distribution. The whole estimation is repeated on $10$ independent
samples, and each point of the figure reports the average, over these $10$
repetitions, of both the execution time of the fit and of the quality indicator
described below. 

\paragraph{Quality indicator: the marginal QQ-RMSE.} To monitor the quality of
the fit we use a normalized marginal \emph{QQ-RMSE} rather than a discrepancy on
the parameters themselves.The reason is twofold. First, the parametrisation
$(\alpha,s)\mapsto \mathbb{P}_{(\alpha,s)}$ is not identifiable on the unconstrained set
(invariance under permutation and atom fusion, Section~\ref{sec:ggc-identifiability}),
so a coordinate-wise error $\|\widehat{\theta}-\theta_0\|$ carries no intrinsic
meaning. Second, any aggregate distance
on the parameters, as well as the final KSD value -- which contains the term
$\langle x,y\rangle$ -- grows mechanically with the dimension $d$, and would
confound a genuine loss of accuracy with a mere increase in dimension. We
therefore measure a distributional quality indicator. The resulting
QQ-RMSE is the quantitative counterpart of the QQ-plots used as a diagnostic in
Figures~\ref{fig:laverny-comparison} and~\ref{fig:application_3d}: it is dimensionless,
and a small value indicates a good fit.

Explicitly, for a configuration $(d,n_{\mathrm{atoms}})$, let $X$ denote the data
sample and $\widehat{X}$ a sample drawn from the fitted model
$\mathcal{G}_{d,n_{\mathrm{atoms}}}(\widehat{\theta}_N)$. Fix the $Q=99$
equally-spaced quantile levels $\tau_1<\dots<\tau_Q$ in $[0.01,0.99]$, and for each
marginal $k\in\{1,\dots,d\}$ denote by $\widehat{q}^{\,X}_k(\tau)$ and
$\widehat{q}^{\,\widehat{X}}_k(\tau)$ the empirical quantile functions of the
$k$-th marginals of $X$ and $\widehat{X}$, and by $\widehat{\sigma}_k$ the
empirical standard deviation of the $k$-th marginal of $X$. The QQ-RMSE of this fit
is
\begin{equation}
  \mathrm{QQ\text{-}RMSE}
  \;=\;
  \frac{1}{d}\sum_{k=1}^{d}
  \sqrt{\frac{1}{Q}\sum_{i=1}^{Q}
    \left(\frac{\widehat{q}^{\,X}_k(\tau_i)-\widehat{q}^{\,\widehat{X}}_k(\tau_i)}
               {\widehat{\sigma}_k}\right)^{\!2}},
  \label{eq:qq-rmse}
\end{equation}
and the value reported at each point of Figure~\ref{fig:time_dimension} is the
average of~\eqref{eq:qq-rmse} over the $10$ repetitions.  \cite{delbarrio1999wasserstein} has shown that QQ-RMSE can be interpreted as an empirical Wassertein distance. 

\paragraph{Results.} Figure~\ref{fig:time_dimension} shows that the execution time
grows essentially linearly with the total number of parameters
$n_{\mathrm{atoms}}(d+1)$. Crucially, the QQ-RMSE stays confined to a narrow band (of order
$0.06$--$0.09$) and remains essentially flat in both $d$ and $n_{\mathrm{atoms}}$:
the statistical performance of the estimator is \emph{stable}, only the
computational cost grows. This confirms that the integral Stein kernel provides a
practically tractable and statistically robust approach to multivariate GGC
estimation, even beyond low-dimensional examples.

\subsection*{Statement on AI use}

All mathematical arguments, results, and proofs in this paper were developed and verified by the authors without the use of artificial intelligence. AI tools were used during the writing process, primarily to suggest occasional improvements to the English syntax and phrasing, and to assist in locating a small number of references. The models used were Fable~5, ChatGPT~5.5 and  Mistral's Vibe (with CNRS enterprise subscription). The suggestions produced by these tools were systematically reviewed and, where appropriate, rewritten and verified by the authors. In particular, no AI system was used to formulate mathematical statements or construct proofs. This
statement is included in accordance with the principles of the
\href{https://leidendeclaration.ai/}{Leiden Declaration on Artificial Intelligence and Mathematics}

%=====================================================================
%  Bibliography
% =====================================================================
\bibliography{references}

\begin{thebibliography}{36}
\providecommand{\natexlab}[1]{#1}
\providecommand{\url}[1]{\texttt{#1}}
\expandafter\ifx\csname urlstyle\endcsname\relax
  \providecommand{\doi}[1]{doi: #1}\else
  \providecommand{\doi}{doi: \begingroup \urlstyle{rm}\Url}\fi

\bibitem[\'Alvarez et~al.(2012)\'Alvarez, Rosasco, and
  Lawrence]{AlvarezRosascoLawrence2012}
Mauricio~A. \'Alvarez, Lorenzo Rosasco, and Neil~D. Lawrence.
\newblock Kernels for vector-valued functions: a review.
\newblock \emph{Foundations and Trends in Machine Learning}, 4\penalty0
  (3):\penalty0 195--266, 2012.
\newblock \doi{10.1561/2200000036}.

\bibitem[Aronszajn(1950)]{Aronszajn1950}
Nachman Aronszajn.
\newblock Theory of reproducing kernels.
\newblock \emph{Transactions of the American Mathematical Society}, 68\penalty0
  (3):\penalty0 337--404, 1950.
\newblock \doi{10.2307/1990404}.

\bibitem[Arras et~al.(2020)Arras, Azmoodeh, Poly, and
  Swan]{ArrasAzmoodehPolySwan2017}
Benjamin Arras, Ehsan Azmoodeh, Guillaume Poly, and Yvik Swan.
\newblock Stein characterizations for linear combinations of gamma random
  variables.
\newblock \emph{Brazilian Journal of Probability and Statistics}, 34\penalty0
  (2):\penalty0 394--413, 2020.
\newblock \doi{10.1214/18-BJPS420}.

\bibitem[Barndorff-Nielsen et~al.(2006)Barndorff-Nielsen, Maejima, and
  Sato]{BarndorffNielsenMaejimaSato2006}
Ole~E. Barndorff-Nielsen, Makoto Maejima, and Ken-iti Sato.
\newblock Some classes of multivariate infinitely divisible distributions
  admitting stochastic integral representations.
\newblock \emph{Bernoulli}, 12\penalty0 (1):\penalty0 1--33, 2006.

\bibitem[Barp et~al.(2019)Barp, Briol, Duncan, Girolami, and
  Mackey]{BarpEtAl2019}
Alessandro Barp, Fran\c{c}ois-Xavier Briol, Andrew~B. Duncan, Mark Girolami,
  and Lester Mackey.
\newblock Minimum {S}tein discrepancy estimators.
\newblock \emph{Advances in Neural Information Processing Systems},
  32:\penalty0 12964--12976, 2019.

\bibitem[Bondesson(1992)]{Bondesson1992}
Lennart Bondesson.
\newblock \emph{Generalized Gamma Convolutions and Related Classes of
  Distributions and Densities}, volume~76 of \emph{Lecture Notes in
  Statistics}.
\newblock Springer-Verlag, New York, 1992.
\newblock \doi{10.1007/978-1-4612-2948-3}.

\bibitem[Bondesson(2009)]{Bondesson2009}
Lennart Bondesson.
\newblock On univariate and bivariate generalized gamma convolutions.
\newblock \emph{Journal of Statistical Planning and Inference}, 139\penalty0
  (11):\penalty0 3759--3765, 2009.
\newblock \doi{10.1016/j.jspi.2009.05.015}.

\bibitem[Carmeli et~al.(2006)Carmeli, De~Vito, and
  Toigo]{CarmeliDeVitoToigo2006}
Claudio Carmeli, Ernesto De~Vito, and Alessandro Toigo.
\newblock Vector valued reproducing kernel {H}ilbert spaces of integrable
  functions and {M}ercer theorem.
\newblock \emph{Analysis and Applications}, 4\penalty0 (4):\penalty0 377--408,
  2006.
\newblock \doi{10.1142/S0219530506000838}.

\bibitem[Chwialkowski et~al.(2016)Chwialkowski, Strathmann, and
  Gretton]{ChwialkowskiStrathmannGretton2016}
Kacper Chwialkowski, Heiko Strathmann, and Arthur Gretton.
\newblock A kernel test of goodness of fit.
\newblock In \emph{Proceedings of the 33rd International Conference on Machine
  Learning (ICML)}, volume~48 of \emph{Proceedings of Machine Learning
  Research}, pages 2606--2615. PMLR, 2016.

\bibitem[Davis and Rabinowitz(1984)]{DavisRabinowitz1984}
Philip~J. Davis and Philip Rabinowitz.
\newblock \emph{Methods of Numerical Integration}.
\newblock Academic Press, Orlando, FL, 2nd edition, 1984.

\bibitem[Del~Barrio et~al.(1999)Del~Barrio, Cuesta-Albertos, Matr\'an, and
  Rodr\'iguez-Rodr\'iguez]{delbarrio1999wasserstein}
Eustasio Del~Barrio, Juan~A. Cuesta-Albertos, Carlos Matr\'an, and Jes\'us~M.
  Rodr\'iguez-Rodr\'iguez.
\newblock Tests of goodness of fit based on the {$L_2$}-{W}asserstein distance.
\newblock \emph{The Annals of Statistics}, 27\penalty0 (4):\penalty0
  1230--1239, 1999.
\newblock \doi{10.1214/aos/1017938923}.

\bibitem[Furman et~al.(2017)Furman, Hackmann, and
  Kuznetsov]{FurmanHackmannKuznetsov2017}
Edward Furman, Daniel Hackmann, and Alexey Kuznetsov.
\newblock On log-normal convolutions: an analytical-numerical method with
  applications to economic capital determination.
\newblock \emph{SSRN Electronic Journal}, 2017.
\newblock \doi{10.2139/ssrn.2924226}.

\bibitem[Gaunt et~al.(2023)Gaunt, Mijoule, and Swan]{GauntMijouleSwan2023}
Robert~E. Gaunt, Guillaume Mijoule, and Yvik Swan.
\newblock An algebra of {S}tein operators.
\newblock \emph{Journal of Mathematical Analysis and Applications},
  521\penalty0 (1):\penalty0 126907, 2023.
\newblock \doi{10.1016/j.jmaa.2022.126907}.

\bibitem[Gorham and Mackey(2017)]{GorhamMackey2017}
Jackson Gorham and Lester Mackey.
\newblock Measuring sample quality with kernels.
\newblock In \emph{Proceedings of the 34th International Conference on Machine
  Learning (ICML)}, volume~70 of \emph{Proceedings of Machine Learning
  Research}, pages 1292--1301. PMLR, 2017.

\bibitem[Gretton et~al.(2012)Gretton, Borgwardt, Rasch, Sch\"olkopf, and
  Smola]{GrettonEtAl2012}
Arthur Gretton, Karsten~M. Borgwardt, Malte~J. Rasch, Bernhard Sch\"olkopf, and
  Alexander Smola.
\newblock A kernel two-sample test.
\newblock \emph{Journal of Machine Learning Research}, 13:\penalty0 723--773,
  2012.

\bibitem[Hoeffding(1948)]{Hoeffding1948}
Wassily Hoeffding.
\newblock A class of statistics with asymptotically normal distribution.
\newblock \emph{The Annals of Mathematical Statistics}, 19\penalty0
  (3):\penalty0 293--325, 1948.
\newblock \doi{10.1214/aoms/1177730196}.

\bibitem[Holmquist(1996)]{Holmquist1996}
Bj\"orn Holmquist.
\newblock The {$d$}-variate vector {H}ermite polynomial of order {$N$}.
\newblock \emph{Linear Algebra and its Applications}, 237--238:\penalty0
  155--190, 1996.
\newblock \doi{10.1016/0024-3795(95)00595-1}.

\bibitem[Laverny(2022)]{Laverny2022Projections}
Oskar Laverny.
\newblock Estimation of high dimensional gamma convolutions through random
  projections.
\newblock 2022.
\newblock arXiv:2203.13741.

\bibitem[Laverny et~al.(2021)Laverny, Masiello, Maume-Deschamps, and
  Rulli\`ere]{LavernyEtAl2021}
Oskar Laverny, Esterina Masiello, V\'eronique Maume-Deschamps, and Didier
  Rulli\`ere.
\newblock Estimation of multivariate generalized gamma convolutions through
  {L}aguerre expansions.
\newblock \emph{Electronic Journal of Statistics}, 15\penalty0 (2):\penalty0
  5158--5202, 2021.
\newblock \doi{10.1214/21-EJS1918}.

\bibitem[Laverny et~al.(2022)Laverny, Ferriero, and
  Nisipasu]{LavernyFerrieroNisipasu2022}
Oskar Laverny, Alessandro Ferriero, and Ecaterina Nisipasu.
\newblock Parametric divisibility of stochastic losses.
\newblock SCOR Paper~43, SCOR SE, 2022.

\bibitem[Liu et~al.(2016)Liu, Lee, and Jordan]{LiuLeeJordan2016}
Qiang Liu, Jason~D. Lee, and Michael~I. Jordan.
\newblock A kernelized {S}tein discrepancy for goodness-of-fit tests.
\newblock In \emph{Proceedings of the 33rd International Conference on Machine
  Learning (ICML)}, volume~48 of \emph{Proceedings of Machine Learning
  Research}, pages 276--284. PMLR, 2016.

\bibitem[Micchelli and Pontil(2005)]{MicchelliPontil2005}
Charles~A. Micchelli and Massimiliano Pontil.
\newblock On learning vector-valued functions.
\newblock \emph{Neural Computation}, 17\penalty0 (1):\penalty0 177--204, 2005.
\newblock \doi{10.1162/0899766052530802}.

\bibitem[Mijoule et~al.(2023)Mijoule, Rai\v{c}, Reinert, and
  Swan]{MijouleRaicReinertSwan2023}
Guillaume Mijoule, Martin Rai\v{c}, Gesine Reinert, and Yvik Swan.
\newblock {S}tein's density method for multivariate continuous distributions.
\newblock \emph{Electronic Journal of Probability}, 28:\penalty0 1--40, 2023.
\newblock \doi{10.1214/23-EJP983}.

\bibitem[Miles(2023)]{Miles2023Thesis}
Justin Miles.
\newblock \emph{On {L}aplace transforms, generalized gamma convolutions, and
  their applications in risk aggregation}.
\newblock Ph.{D}.\ thesis, York University, Toronto, Ontario, 2023.

\bibitem[Miles et~al.(2021)Miles, Furman, and
  Kuznetsov]{MilesFurmanKuznetsov2021}
Justin Miles, Edward Furman, and Alexey Kuznetsov.
\newblock Risk aggregation: a general approach via the class of generalized
  gamma convolutions.
\newblock \emph{Variance}, 13\penalty0 (2):\penalty0 233--249, 2021.

\bibitem[Muandet et~al.(2017)Muandet, Fukumizu, Sriperumbudur, and
  Sch\"olkopf]{MuandetEtAl2017}
Krikamol Muandet, Kenji Fukumizu, Bharath Sriperumbudur, and Bernhard
  Sch\"olkopf.
\newblock Kernel mean embedding of distributions: a review and beyond.
\newblock \emph{Foundations and Trends in Machine Learning}, 10\penalty0
  (1--2):\penalty0 1--141, 2017.
\newblock \doi{10.1561/2200000060}.

\bibitem[Sayit(2024)]{Sayit2024}
Hasanjan Sayit.
\newblock Weak convergence implies convergence in mean within {GGC}.
\newblock 2024.
\newblock arXiv:2407.15105.

\bibitem[Sch\"olkopf and Smola(2002)]{ScholkopfSmola2002}
Bernhard Sch\"olkopf and Alexander~J. Smola.
\newblock \emph{Learning with Kernels: Support Vector Machines, Regularization,
  Optimization, and Beyond}.
\newblock MIT Press, Cambridge, MA, 2002.

\bibitem[Smola et~al.(2007)Smola, Gretton, Song, and
  Sch\"olkopf]{SmolaEtAl2007}
Alexander~J. Smola, Arthur Gretton, Le~Song, and Bernhard Sch\"olkopf.
\newblock A {H}ilbert space embedding for distributions.
\newblock In \emph{Algorithmic Learning Theory (ALT 2007)}, volume 4754 of
  \emph{Lecture Notes in Computer Science}, pages 13--31. Springer, 2007.
\newblock \doi{10.1007/978-3-540-75225-7_5}.

\bibitem[Sriperumbudur et~al.(2010)Sriperumbudur, Gretton, Fukumizu,
  Sch\"olkopf, and Lanckriet]{Sriperumbudur2010}
Bharath~K. Sriperumbudur, Arthur Gretton, Kenji Fukumizu, Bernhard Sch\"olkopf,
  and Gert R.~G. Lanckriet.
\newblock Hilbert space embeddings and metrics on probability measures.
\newblock \emph{Journal of Machine Learning Research}, 11:\penalty0 1517--1561,
  2010.

\bibitem[Stein(1972)]{Stein1972}
Charles Stein.
\newblock A bound for the error in the normal approximation to the distribution
  of a sum of dependent random variables.
\newblock In \emph{Proceedings of the Sixth Berkeley Symposium on Mathematical
  Statistics and Probability, Volume~2}, pages 583--602, Berkeley, 1972.
  University of California Press.

\bibitem[Stein(1986)]{Stein1986}
Charles Stein.
\newblock \emph{Approximate Computation of Expectations}, volume~7 of
  \emph{Lecture Notes--Monograph Series}.
\newblock Institute of Mathematical Statistics, Hayward, CA, 1986.

\bibitem[Steinwart and Christmann(2008)]{SteinwartChristmann2008}
Ingo Steinwart and Andreas Christmann.
\newblock \emph{Support Vector Machines}.
\newblock Information Science and Statistics. Springer, New York, 2008.
\newblock \doi{10.1007/978-0-387-77242-4}.

\bibitem[Thorin(1977{\natexlab{a}})]{Thorin1977Lognormal}
Olof Thorin.
\newblock On the infinite divisibility of the lognormal distribution.
\newblock \emph{Scandinavian Actuarial Journal}, 1977\penalty0 (3):\penalty0
  121--148, 1977{\natexlab{a}}.
\newblock \doi{10.1080/03461238.1977.10405635}.

\bibitem[Thorin(1977{\natexlab{b}})]{Thorin1977Pareto}
Olof Thorin.
\newblock On the infinite divisibility of the {P}areto distribution.
\newblock \emph{Scandinavian Actuarial Journal}, 1977\penalty0 (1):\penalty0
  31--40, 1977{\natexlab{b}}.
\newblock \doi{10.1080/03461238.1977.10405623}.

\bibitem[Withers(2000)]{Withers2000}
Christopher~S. Withers.
\newblock A simple expression for the multivariate {H}ermite polynomials.
\newblock \emph{Statistics \& Probability Letters}, 47\penalty0 (2):\penalty0
  165--169, 2000.
\newblock \doi{10.1016/S0167-7152(99)00153-4}.

\end{thebibliography}

% ---------------------------------------------------------------------
\appendix
\refstepcounter{section}
\renewcommand{\thesubsection}{\thesection-\arabic{subsection}}
\section*{Appendix}

This appendix contains some more technical results.
% ---------------------------------------------------------------------
\subsection{Proof of Proposition~\ref{prop:equivalence-gdn} (equivalence on $\mathcal{G}_{d,n}$)}
\label{app:proof-equivalence-gdn}
% ---------------------------------------------------------------------
 
We start from the integral form~\eqref{eq:stein-ggc-discrete} and show
that it can be rewritten as the polynomial-coefficient
form~\eqref{eq:stein-arras-discrete}. Fix $j \in \{1,\dots,n\}$ and
consider
$ \displaystyle I_{j}(x) \coloneqq \int_{0}^{\infty}\! e^{-z}\,\ip{s_{j}}{f(x + z s_{j})}\,\mathrm{d}z$.
Using the Taylor expansion of $f$ at $x$,
\[
  f(x + z s_{j})
  \;=\; \sum_{m=0}^{\infty}\!\frac{z^{m}}{m!}\,(\ip{s_{j}}{\nabla_{x}})^{m} f(x)
  \;=\; \sum_{m=0}^{\infty}\!\frac{z^{m}}{m!}\,(D_{j}^{*})^{m} f(x),
\]
with $D_{j}^{*} \coloneqq \ip{s_{j}}{\nabla_{x}}$. The exchange of
$\displaystyle \int_{0}^{\infty}\! e^{-z}\dots\mathrm{d}z$ and the Taylor sum is by
justified the identity
$ \displaystyle \int_{0}^{\infty}\! z^{m} e^{-z}\,\mathrm{d}z = m!$. Hence
\[
  I_{j}(x)
  \;=\; \sum_{m=0}^{\infty}\!\ip{s_{j}}{(D_{j}^{*})^{m} f(x)}
  \;=\; \ip{s_{j}}{(I - D_{j}^{*})^{-1} f(x)}.
\]
Substituting in~\eqref{eq:stein-ggc-discrete},
\begin{equation}
  (\mathcal{T}^{\mathrm{GGC}}_{X} f)(x)
  \;=\; \ip{x}{f(x)} \,-\, \sum_{j=1}^{n}\!\alpha_{j}\,\ip{s_{j}}{(I - D_{j}^{*})^{-1} f(x)}.
  \label{eq:proof-eq-1}
\end{equation}
 
Multiplying both sides of~\eqref{eq:proof-eq-1} by the formal product
operator $\prod_{j=1}^{n}(I - D_{j}^{*})$ and using
$\prod_{j=1}^{n}(I - D_{j}^{*}) \cdot (I - D_{j}^{*})^{-1} = \prod_{k \neq j}(I - D_{k}^{*})$,
\begin{equation}
  \prod_{j=1}^{n}(I - D_{j}^{*})\,(\mathcal{T}^{\mathrm{GGC}}_{X} f)(x)
  \;=\; \Bigl\langle x,\,{\textstyle\prod_{j=1}^{n}}(I - D_{j}^{*})\, f(x)\Bigr\rangle
       - \sum_{j=1}^{n}\!\alpha_{j}\,\Bigl\langle s_{j},\,{\textstyle\prod_{k \neq j}}(I - D_{k}^{*})\, f(x)\Bigr\rangle,
  \label{eq:proof-eq-2}
\end{equation}
which is exactly $\ip{x}{A(D) f(x)} - \ip{B(D)}{f(x)}$ with $A$ and $B$
given by~\eqref{eq:AB-discrete} -- i.e.\ the Arras polynomial form of
Proposition~\ref{lem:arras-multi}.
 
Expanding the products $\prod_{j=1}^{n}(I - D_{j}^{*})$ and
$\prod_{k \neq j}(I - D_{k}^{*})$ via the elementary symmetric
polynomials,
\[
  \prod_{j=1}^{n}(I - D_{j}^{*})
  \;=\; \sum_{\ell=0}^{n}(-1)^{\ell}\, e_{\ell}(D^{*}),
  \qquad
  \prod_{k \neq j}(I - D_{k}^{*})
  \;=\; \sum_{\ell=0}^{n-1}(-1)^{\ell}\, e_{\ell}((D^{*})_{j}),
\]
substituting in~\eqref{eq:proof-eq-2}, separating the orders $\ell = 0$,
$\ell = 1, \dots, n-1$ and $\ell = n$, and rearranging (we add and subtract
$\sum_{j=1}^{n}\alpha_{j}\,\ip{s_{j}}{e_{\ell}(D^{*})f(x)}$ at each
intermediate order), one obtains the explicit
form~\eqref{eq:stein-arras-discrete}. Both
$\mathcal{T}^{\mathrm{GGC}}_{X}$ and $\mathcal{T}^{\mathrm{Arras}}_{X}$ thus
coincide on $\Sclass(\R^{d}, \R^{d})$. \qed
\paragraph{Concise form of the polynomial stein kernel (RBF).}
Equation~\eqref{eq:stein-arras-discrete} can be rewritten as
\begin{equation}
  (\mathcal{T}^{\mathrm{Arras}}_{X} f)(x)
  \;=\; \sum_{l=0}^{n}(-1)^{l}\!\left[
      \ip{x}{e_{l}(D^{*}) f(x)}
      - \sum_{j=1}^{n}\alpha_{j}\,\bigl\langle s_{j},\,(2 e_{l}(D^{*}) - e_{l}((D^{*})_{j}))\, f(x)\bigr\rangle
  \right],
  \label{eq:poly-op-compact}
\end{equation}
with $D_{j}^{*} = \ip{s_{j}}{\nabla_{x}}$, $D^{*} = (D_{1}^{*}, \dots, D_{n}^{*})$, the convention
$e_{l}(D^{*}) = 0$ for $l > n$, and $e_{l}((D^{*})_{j})$ the $\ell$-th elementary symmetric
polynomial in $D^{*}$ with the $j$-th entry omitted. Following~\citet{Holmquist1996,Withers2000},
the multivariate Hermite polynomial of multi-index $\beta \in \N^{n}$ attached to $v \in \R^{n}$
and $\Omega \in \R^{n \times n}$,
\begin{equation*}
  \mathbb{H}_{\beta}(v, \Omega)
  \;\coloneqq\;
  \E\!\left[\,\prod_{j=1}^{n} (v_{j} + \im\, Z_{j})^{\beta_{j}}\,\right],
  \qquad Z \sim \mathcal{N}(0_{n}, \Omega),
\end{equation*}
has moment-generating function
\begin{equation}
  \sum_{\beta \in \N^{n}} \frac{t^{\beta}}{\beta!}\,\mathbb{H}_{\beta}(v, \Omega)
  \;=\; \exp\!\Bigl(\,t^{\top} v - \tfrac{1}{2}\, t^{\top}\Omega\, t\,\Bigr),
  \qquad t \in \R^{n},
  \label{eq:hermite-mgf}
\end{equation}
an immediate consequence of $\E[e^{\im\,\ip{t}{Z}}] = e^{-\frac{1}{2} t^{\top}\Omega t}$
($t^{\beta} = t_{1}^{\beta_{1}}\cdots t_{n}^{\beta_{n}}$, $\beta! = \beta_{1}!\cdots\beta_{n}!$).

\paragraph{Identification $D_{\beta}^{*} k(x,y) = \mathbb{H}_{\beta}(v,\Omega)\, k(x,y)$.}
For a shift $h = \sum_{j} t_{j}\, s_{j}$, the Taylor expansion of $k(\cdot, y)$ at $x$ together with
$(h \cdot \nabla_{x})^{m} = (\sum_{j} t_{j} D_{j}^{*})^{m}$ gives
\begin{equation}
  k(x + h,\, y)
  \;=\; \sum_{\beta \in \N^{n}}\!\frac{t^{\beta}}{\beta!}\,
        \prod_{j=1}^{n}\bigl(\!\ip{s_{j}}{\nabla_{x}}\!\bigr)^{\beta_{j}} k(x, y),
  \label{eq:taylor-shift}
\end{equation}
while the Gaussian-conjugate identity
$\|(x - y) + h\|^{2} = \|x - y\|^{2} + 2\ip{x - y}{h} + \|h\|^{2}$ yields
\begin{equation*}
  k(x + h,\, y) \;=\; k(x, y)\,\exp\!\Bigl(\,t^{\top} v - \tfrac{1}{2}\,t^{\top} \Omega\, t\,\Bigr),
  \qquad
  v_{j} \coloneqq -\frac{\ip{s_{j}}{x - y}}{\sigma^{2}},\;
  \Omega_{ij} \coloneqq -\frac{\ip{s_{i}}{s_{j}}}{\sigma^{2}}.
\end{equation*}
Matching~\eqref{eq:taylor-shift} with $k(x,y)$ times~\eqref{eq:hermite-mgf} term by term in $\beta$,
\begin{equation}
  D_{\beta}^{*}\, k(x, y)
  \;=\;
  \mathbb{H}_{\beta}(v, \Omega)\, k(x, y).
  \label{eq:Dstar-Hermite}
\end{equation}

\paragraph{Action of $e_{l}(D^{*})$ and the full Stein kernel.}
Summing~\eqref{eq:Dstar-Hermite} over $\beta \in \{0,1\}^{n}$ with $|\beta| = l$ and writing
\begin{equation*}
  \mathbb{H}^{\star}_{l}(v, \Omega)
  \coloneqq\!\!
  \sum_{\substack{\beta \in \{0,1\}^{n} \\ |\beta| = l}}\!\! \mathbb{H}_{\beta}(v, \Omega),
  \qquad
  \mathbb{H}^{\star,-j}_{l}(v, \Omega)
  \coloneqq\!\!
  \sum_{\substack{\beta \in \{0,1\}^{n},\,\beta_{j} = 0 \\ |\beta| = l}}\!\! \mathbb{H}_{\beta}(v, \Omega),
\end{equation*}
gives $e_{l}(D^{*})^{(x)} k = \mathbb{H}^{\star}_{l}\, k$ and
$e_{l}((D^{*})_{j})^{(x)} k = \mathbb{H}^{\star,-j}_{l}\, k$. Since $\nabla_{x} k = -\nabla_{y} k$,
acting on the $y$-variable adds a factor $(-1)^{l}$; composing the two variables,
$e_{l}(D^{*})^{(y)} e_{r}(D^{*})^{(x)} k = (-1)^{l}\,\mathbb{H}^{\star}_{l+r}\, k$, and likewise
$e_{l}((D^{*})_{j})^{(y)} e_{r}((D^{*})_{p})^{(x)} k = (-1)^{l}\,\mathbb{H}^{\star}_{l-j+r-p}\, k$
with the $-j,-p$ exclusions on each side. Applying the compact
operator~\eqref{eq:poly-op-compact} successively on the $x$- and $y$-arguments of
$K(x,y) = B\, k(x,y)$ and collecting the four cross-products yields
\begin{equation*}
\begin{aligned}
  k_{0,\theta}^{\mathrm{Arras}}(x, y)
  \;=\; k(x, y)\!\sum_{l, r = 0}^{n} (-1)^{l}\Bigl[
      &x^{\top} B\, y\,\mathbb{H}^{\star}_{l + r}(v, \Omega) \\[-2pt]
      &-\sum_{j=1}^{n}\alpha_{j}\,(s_{j}^{\top} B\, y)\,\bigl(2\mathbb{H}^{\star}_{l+r}(v,\Omega) - \mathbb{H}^{\star}_{l-j+r}(v,\Omega)\bigr) \\
      &-\sum_{p=1}^{n}\alpha_{p}\,(s_{p}^{\top} B\, x)\,\bigl(2\mathbb{H}^{\star}_{l+r}(v,\Omega) - \mathbb{H}^{\star}_{l+r-p}(v,\Omega)\bigr) \\
      &+\sum_{j, p = 1}^{n}\alpha_{j}\,\alpha_{p}\,(s_{j}^{\top} B\, s_{p})\,\bigl(4\mathbb{H}^{\star}_{l+r}(v,\Omega) - 2\mathbb{H}^{\star}_{l-j+r}(v,\Omega) \\[-2pt]
      &\hphantom{+\sum_{j, p = 1}^{n}\alpha_{j}\,\alpha_{p}\,(s_{j}^{\top} B\, s_{p})\,\bigl(}{} - 2\mathbb{H}^{\star}_{l+r-p}(v,\Omega) + \mathbb{H}^{\star}_{l-j+r-p}(v,\Omega)\bigr)
  \Bigr].
\end{aligned}
\end{equation*}
For the isotropic choice $B = I_{d}$ (Section~\ref{sec:num-laverny-comparison}), the quadratic
forms $x^{\top} B\, y$, $s_{j}^{\top} B\, y$, $s_{p}^{\top} B\, x$, $s_{j}^{\top} B\, s_{p}$ reduce
to $\ip{x}{y}$, $\ip{s_{j}}{y}$, $\ip{s_{p}}{x}$, $\ip{s_{j}}{s_{p}}$.
For numerical evaluation, we use the recursion from \cite{Holmquist1996}:
\begin{equation*}
  \mathbb{H}_{\beta + e_{i}}(v, \Omega)
  \;=\; v_{i}\,\mathbb{H}_{\beta}(v, \Omega)
       - \sum_{j=1}^{n} \beta_{j}\, \Omega_{ji}\,\mathbb{H}_{\beta - e_{j}}(v, \Omega).
\end{equation*}
% ---------------------------------------------------------------------
\subsection{Closed form of the single-jump Gaussian expectation}
\label{app:gaussian-single-jump}
% ---------------------------------------------------------------------

We derive the closed-form expression, announced in
Section~\ref{sec:numerical-applications}, of the single-jump
expectations appearing in the discrete Stein
kernel~\eqref{eq:ggc-stein-kernel-discrete} when $k$ is the Gaussian
(RBF) kernel
\(
  k(x,y) = \exp\!\bigl(-\|x-y\|^{2}/(2\sigma^{2})\bigr)
\)
of~\eqref{eq:rbf-kernel}. Throughout, $E \sim \mathrm{Exp}(1)$ and
$s \in \Rp^{d} \setminus \{0\}$ is a fixed atom of the Thorin measure.
We write $\mathrm{erfc}$ for the complementary error function and
$\mathrm{erfcx}(t)$ for its scaled
(numerically stable) version,
\[
  \mathrm{erfc}(t) \;=\; \frac{2}{\sqrt{\pi}}\!\int_{t}^{\infty}\! e^{-u^{2}}\,\mathrm{d}u,
  \qquad
  \mathrm{erfcx}(t) \;=\; e^{t^{2}}\,\mathrm{erfc}(t).
\]

\paragraph{Reduction to a half-line Gaussian integral.}
Since $E$ has density $z \mapsto e^{-z}$ on $(0,\infty)$,
\[
  \E\bigl[k(x + s\, E,\, y)\bigr]
  \;=\; \int_{0}^{\infty}\! e^{-z}\,
        \exp\!\left(-\frac{\|x - y + z s\|^{2}}{2\sigma^{2}}\right)\mathrm{d}z.
\]
It can be written as:
\begin{equation}
  \E\bigl[k(x + s\, E,\, y)\bigr]
  \;=\; k(x,y)\!\int_{0}^{\infty}\! \exp\!\bigl(-a\, z^{2} - c\, z\bigr)\,\mathrm{d}z,
  \qquad
  a \coloneqq \frac{\|s\|^{2}}{2\sigma^{2}},
  \quad
  c \coloneqq 1 + \frac{\ip{x-y}{s}}{\sigma^{2}},
  \label{eq:app-gauss-reduced}
\end{equation}
where the coefficient $a > 0$ because $s \neq 0$.
The integral can be rewritten as:
\[
  \int_{0}^{\infty}\! e^{-a z^{2} - c z}\,\mathrm{d}z
  \;=\; \frac{1}{2}\sqrt{\frac{\pi}{a}}\,
        \exp\!\Bigl(\frac{c^{2}}{4a}\Bigr)\,
        \mathrm{erfc}\!\Bigl(\frac{c}{2\sqrt{a}}\Bigr)
  \;=\; \frac{1}{2}\sqrt{\frac{\pi}{a}}\,
        \mathrm{erfcx}\!\Bigl(\frac{c}{2\sqrt{a}}\Bigr),
\]

\paragraph{Closed-form formulas.}
Substituting $a$ and $c$ from~\eqref{eq:app-gauss-reduced} back into the
identity, and noting that
$\tfrac{c}{2\sqrt{a}} = \bigl(\sigma^{2} + \ip{x-y}{s}\bigr)/(\sqrt{2}\,\sigma\|s\|)$,
we obtain
\begin{equation*}
  \E\bigl[k(x + s\, E,\, y)\bigr]
  \;=\; k(x,y)\,\sqrt{\frac{\pi}{2}}\,\frac{\sigma}{\|s\|}\,
        \mathrm{erfcx}\!\left(\frac{\sigma^{2} + \ip{x-y}{s}}{\sqrt{2}\,\sigma\,\|s\|}\right)
\end{equation*}
By symmetry of the kernel -- replacing $x-y$ by $-(x-y)$, i.e.\
$\ip{x-y}{s} \to -\ip{x-y}{s}$ -- the companion single-jump expectation
writes
\begin{equation*}
  \E\bigl[k(x,\, y + s\, E)\bigr]
  \;=\; k(x,y)\,\sqrt{\frac{\pi}{2}}\,\frac{\sigma}{\|s\|}\,
        \mathrm{erfcx}\!\left(\frac{\sigma^{2} - \ip{x-y}{s}}{\sqrt{2}\,\sigma\,\|s\|}\right).
\end{equation*}

\end{document}